\documentclass{amsart}

\usepackage{amssymb}
\usepackage{amsmath}
\usepackage{amsthm}
\usepackage{amsbsy}
\usepackage{bm}
\usepackage{hyperref}
\date{\today}
\usepackage{cite}
\usepackage{array}
\usepackage{graphicx}
\usepackage{float}
\usepackage{upgreek}
\usepackage{color}
\usepackage{amsfonts}
\usepackage{mathrsfs}

\newcommand{\Prob}{\mathbb{P}}
\newcommand{\Z}{\mathbb{Z}}
\newcommand{\E}{\mathbb{E}}
\newcommand{\m}{\bf m}

    \newtheorem{theorem}{Theorem}
    \newtheorem{lemma}{Lemma}
    
    \newtheorem{proposition}{Proposition}

 \theoremstyle{definition} % For roman text in the body
    
    \newtheorem{fact}{Fact}
    
    \newtheorem{remark}{Remark}
    \newtheorem{example}[theorem]{Example}
    \newtheorem{exercise}[theorem]{Exercise}
    
    \newtheorem{assumption}{Assumption}

\def\newblock{\hskip .11em plus .33em minus .07em}
\usepackage[utf8]{inputenc}

\def\suchthat{\; : \;}

\def\Z{\mathbb{Z}}

\def\N{\mathbb{N}}

\def\R{\mathbb{R}}

\def\Z{\mathbb{Z}}

\def\<{\langle}
\def\>{\rangle}

\def\given{\left.\vphantom{\hbox{\Large (}}\right|}

\newcommand\mnote[1]{} %off
\newcommand\be{\begin{equation*}}

\newcommand\ee{\end{equation*}}

\newcommand\ben{\begin{equation}}
\newcommand\een{\end{equation}}
\newcommand\bes{\begin{eqnarray*}}
\newcommand\ees{\end{eqnarray*}}

\newcommand\bex{\begin{exercise}}
\newcommand\eex{\end{exercise}}
\newcommand\beg{\begin{example}}
\newcommand\eeg{\end{example}}
\newcommand\benu{\begin{enumerate}}
\newcommand\eenu{\end{enumerate}}
\newcommand\beit{\begin{itemize}}
\newcommand\eeit{\end{itemize}}
\newcommand\berk{\begin{remark}}
\newcommand\eerk{\end{remark}}
\newcommand\bdefn{\begin{defintion}}
\newcommand\edefn{\end{definition}}
\newcommand\bthm{\begin{theorem}}
\newcommand\ethm{\end{theorem}}
\newcommand\bprf{\begin{proof}}
\newcommand\eprf{\end{proof}}
\newcommand\blem{\begin{lemma}}
\newcommand\elem{\end{lemma}}

\newcommand{\Poi}{\mbox{\rm Poi}}

\newcommand{\sm}{{\raise0.3ex\hbox{$\scriptstyle \setminus$}}}

\def\l{\left}
\def\r{\right}

\def\CHI{\mathchoice%
{\raise2pt\hbox{$\chi$}}%
{\raise2pt\hbox{$\chi$}}%
{\raise1.3pt\hbox{$\scriptstyle\chi$}}%
{\raise0.8pt\hbox{$\scriptscriptstyle\chi$}}}
\def\smalloplus{\raise1pt\hbox{$\,\scriptstyle \oplus\;$}}

\usepackage{graphicx}
\usepackage{amssymb,amsmath,amsfonts,amsthm}
\usepackage[hyperpageref]{backref}
\renewcommand*{\backref}[1]{}
\usepackage[capitalise,noabbrev]{cleveref}
\usepackage{tcolorbox,longfbox}
\allowdisplaybreaks
\usepackage{epstopdf}
\usepackage{hyperref}
\usepackage{etoolbox}
\usepackage{float}
\usepackage{enumerate}
\usepackage{mathtools}
\usepackage{comment}
\usepackage[bottom]{footmisc}
\usepackage{tikz}
\usepackage{bbm}
\title[On the diameter of random uniform hypergraphs]{On the diameter of random uniform 
hypergraphs in dense regime}

\author{Kartick Adhikari}
\address{Department of Mathematics, Indian Institute of Science Education and Research, Bhopal 462066}

\email{kartick [at] iiserb.ac.in}

\author{Asrafunnesa Khatun}
\address{Department of Mathematics, Indian Institute of Science Education and Research, Bhopal 462066}
\email{asrafunnesa23 [at] iiserb.ac.in}
\date{\today}

\begin{document}

\newcommand{\acr}{\newline\indent}

	\begin{abstract}
		
%		In \cite{bollobas1981diameter}, showed that the diameter of the Erd\H {o}s-R\'{e}nyi graph 
%		concentrated only on two points. In this article we give an alternative proof of the Bollob{\'a}s 
%		result using the Chen-Stein method. This method can potentially be used to prove various problem in this field.
		
		For a fixed natural number $t\ge 2$, we consider  $t$-uniform  random hypergraphs $\mathscr{H} (n,t,p)$ on $n$ vertices $[n]=\{1,\ldots, n\}$, where  each $t$-subset of $[n]$ is included as a hyperedge with probability $p$ and independently. We show that the diameter of $\mathscr{H} (n,t,p)$ is concentrated only at two points in the dense regime. More precisely, suppose $diam(\mathcal H)$ denotes the diameter of a hypergraph $\mathcal H$ on $n$ vertices. We show that, for fixed  $t,c,d$ constants, if $n$ and $p$ (depends on $t,c,d,n$) satisfy
		\[ 
		\frac{ (t-1)^ {d} N^{d} p^{d}} {n}= \log \left( \frac{n^2}{c} \right), \mbox{ where } N={n-1\choose t-1},
		\]
$c$ is a positive constant and $ d\ge 2$ is a natural number,	then 
		\[ 
		\lim_{n \to \infty} \Prob \left( diam( \mathcal{H}) = d \right) = e^{- \frac{c}{2}} \text{ and } \lim_{n \to \infty} \Prob \left( diam(\mathcal{H}) = d+1 \right) = 1- e^{- \frac{c}{2}}.
		\]
		In particular, the case where \( t = 2 \) corresponds to the diameter of the Erdős-Rényi graph, as established by Bollob\'as in \cite[Theorem~6]{bollobas1981diameter}. Bollob\' as's result was proven using the moments method, which is challenging  to apply in our context due to the complexity of the model. In this paper, we utilize the Stein-Chen method along with coupling techniques to prove our result. This approach can potentially be used to solve various problems, in particular diameter problems, in more complex networks.
	
	\end{abstract}
	\maketitle
	
	\section{Introduction and the main result}
	The graph was first introduced by Euler in 1735 to solve the K\"{o}nigsberg bridge problem \cite{alexanderson2006cover}. A finite graph $G$ is a pair $\left( V(G), E(G) \right)$, where $V(G)$ is a finite set, called the set of vertices, and $E(G)$ is a subset of pairs of distinct elements of $V(G)$, called the set of edges. In a graph, a path of length $k$ (a natural number) between two vertices is an alternating sequence of $k+1$ distinct vertices and $k$ edges, which starts at one of those two vertices, ends at the other, and each edge connects the preceding vertex to the following one. In other words, if $x, y \in V(G)$, then a path of length $k$ from $x$ to $y$ can be written as 
	\[
	\left( x=x_0, e_1, x_1, e_2,\ldots, e_k, x_k =y \right),
	\]
	where $x_i, 0 \leq i \leq k$, are distinct vertices and each edge $e_i$ connects the vertices $x_{i-1}$ and $x_i$. A graph is called connected if there exists a path between any two vertices, otherwise disconnected. A shortest path between two vertices is a path with the minimum length among all possible paths connecting them. The distance between two vertices $x$ and $y$ in $G$, denoted by $ d_G (x,y) $, is the length of a shortest path. By convention, $d_G (x,x)=0$ for any vertex $x$. Moreover, if there is no path between $x$ and $y$, then we consider $d_G (x,y) = \infty$. The diameter of $G$, denoted by $diam(G)$, is the length of the longest among all shortest paths between any two nodes, that is,
	\[
	diam(G)= \max \left\{ d_G(x,y) : x,y \in V(G) \right\}.
	\]
	Clearly, $diam(G) = \infty$ if $G$ is disconnected. The diameter provides insight into the networks. In partiocular, it helps to investigate the topological properties and connectivity of the networks. For example, the connectivity of  the world-wide web was studied in \cite{albert1999diameter}.  On the other hand, a small diameter indicates rapid spreading of infectious diseases \cite{watts1998collective}.
	
	Significantly, the diameter of random graphs enhances our understanding of the structural properties and behavior of random networks, for example, the efficiency of information spreading, the spread of epidemiology \cite{mckee2024structural} and network's robustness \cite{callaway2000network}. Around 1960, the random graph theory was introduced by Erd\H{o}s and R\'{e}nyi through a series of seminal papers \cite{erdds1959random}, \cite{erdHos1960evolution}. These random graphs are widely utilized to model real-world networks such as social, information, and biological networks \cite{newman2003random}, \cite{newman2003structure}. However, real-networks often differ from Erd\H{o}s-R\'{e}nyi models.  
	
	The Erd\H{o}s-R\'{e}nyi $\mathscr{G}(n,p)$ random graph model is defined on the vertex set $ [n] = \{1,2,\ldots ,n \}$ and each possible edge is included independently with probability $p$, $0 < p <1$. In other words, $\mathscr{G}(n,p)$ is the probability space $\left( \mathscr{G} [n], \mathscr{F_n}, \Prob \right) $ \cite{Bollobás_2001}, where $\mathscr{G} [n]$ denote the set of all graphs with vertex set $[n]$, $\mathscr{F}$ is the $\sigma$-algebra of all subsets of $\mathscr{G} [n]$, and probability $\Prob $ assigned as follows:
	if $G \in \mathscr{G} [n]$ consists $m$ edges, then  
	\[
	\Prob (G) = p^m \left(1-p \right)^{ \binom {n}{2} -m}.
	\]
	The study of the diameter of the Erd\H{o}s-R\'{e}nyi random graphs was initiated by Klee and Larman \cite{klee1981diameters}, and Bollob{\'a}s \cite{bollobas1981diameter}, \cite{bollobas1982diameter}. Also, the diameters of various types of random graphs, such as scale free random graph  \cite{bollobas2004diameter}, weighted random graphs \cite{amini2015diameter}, hyperbolic random graphs \cite{friedrich2018diameter}, random geometric graphs in the unit ball \cite{ellis2007random}, and random planner graphs \cite{chapuy2015diameter} were extensively studied.
	
	In \cite[Theorem 6]{bollobas1981diameter} Bollob{\'a}s proved the following theorem under the following assumption. We use the notation $\N$ for the set of natural numbers.
	
\begin{assumption}\label{ass bollobas}
Let $c$ be a  positive constant and $d\in \N$ with $d \geq 2$. Consider the parameters $n\in \N$ and $p\in (0,1)$ (depending on $c, d, n$) such that
\[
p^d n^{d-1} = \log \left( \frac{n^2}{c} \right).
\]
\end{assumption} 
\noindent Observe that if $d$ is fixed then $ {pn}/(\log n)^3 \to \infty$,  as $ n \to \infty$. If $d=d(n)$ varies with $n$ then this additional condition is required as in \cite{bollobas1981diameter}. For brevity, we assume $d$ is a fixed positive integer. However, our proof techniques work if $d$ is function of $n$ as in \cite{bollobas1981diameter}. We write  $ G \in \mathscr{G}(n,p)$ to mean that $G$ is an Erd\H{o}s-R\'enyi graph with distribution  $\mathscr{G}(n,p)$. Observe that $G$ depends on $n$, however, for ease of writing  we suppress $n$ in  $G$.

	\begin{theorem}[\cite{bollobas1981diameter}]\label{Bollobas theorem}
	Let $c,d,n,p$  be as in Assumption \ref{ass bollobas}, and $G \in \mathscr{G}(n,p)$. Then
	\[ 
	\lim_{n \to \infty} \Prob \left( diam( G) = d \right) = e^{- \frac{c}{2}} \text{ and } \lim_{n \to \infty} \Prob \left( diam(G) = d+1 \right) = 1- e^{- \frac{c}{2}}.
	\]	
   \end{theorem}

The proof of this reult in \cite{bollobas1981diameter} mainly relies on the method of moments.  However, calculating the diameter of the most natural generalisations of Erd\H os-R\'enyi graphs, namely the random uniform hypergraphs and the Linial Meshulam complex \cite{linialmeshulam},\cite{meshulam2009homological}, using this method will be challenging due to the depedency in the model. In this paper, we first provide an alternative proof of  \cref{Bollobas theorem} using the Chernoff bound, the Stein-Chen method, and the coupling technique. Then, using this approach we calculate the diameter of the random uniform hypergraphs. Moreover, our approach can potentially be used in calculating the diameter for more general structure, for example, the Linial-Meshulam complex and the $r$-set line graphs of random uniform hypergraphs \cite{adhikari2025spectrum}.

	A finite hypergraph $\mathcal{H}$ is a pair $(V(\mathcal{H}), \mathcal{E} (\mathcal{H}))$, where $V(\mathcal{H})$ represents a finite set called the set of vertices of $\mathcal{H}$ and $\mathcal{E} (\mathcal{H})$ is a collection of non-empty subsets of $V(\mathcal{H})$ called hyperedges of $\mathcal{H}$ \cite{bretto2013hypergraph}. Throughout, we assume $t \geq 2$ is an integer. A $t$-uniform hypergraph is a hypergraph in which each hyperedge has size $t$. Note that the $2$-uniform hypergraph coincides with the notion of graph. Let $x,y \in V(\mathcal{H}) $, then a path of length $k$ from $x$ to $y$ in $\mathcal{H}$ is a vertex-edge alternative sequence  
	\[
     \left( x=x_0, e_1,x_1,e_2,\ldots ,e_k, x_k=y \right),
	\]
	 where $x_i , 0 \leq i \leq k$, are distinct vertices, $ e_{i}, 1\leq  i \leq k$, are distinct hyperedges, and $e_i$ contains the consecutive vertices $x_{i-1}$ and $x_i$. The definitions of connectivity, a shortest path and distance extend naturally to hypergraph from graph. To avoid repetition, we are excluding these definitions. Let $ d_\mathcal{H} (x,y) $ denote the distance between vertices $x$ and $y$ in $\mathcal H$, and $d_\mathcal{H} (x,y) = \infty$ if there is no path connecting $x$ and $y$. The diameter of $\mathcal{H}$, denoted by $ diam(\mathcal{H})$, is given by 
	 \[
	 diam(\mathcal{H})= \max \left\{ d_\mathcal{H} (x,y) : x,y \in V(\mathcal{H}) \right\}.
	 \]
	 It is clear that if $\mathcal{H}$ is disconnected then $diam(\mathcal{H}) = \infty$.
	 Hypergraphs capture higher-order interactions in the real-world such as group chats, co-authorship on research papers, protein complexes, and chemical reactions, more accurately than graphs, which naturally model pairwise interactions. It is used to model social, neural, ecological and biological networks, where higher-order interactions are very common \cite{ghoshal2009random}, \cite{chen2025simple}, \cite{kook2020evolution}. 
	
	Random hypertrees and hyperforests \cite{lyamin1974random}, connectivity of random hypergraphs \cite{bergman2024connectivity}
	components in random hypergraph \cite{cooley2018largest}, \cite{cooley2018size}, \cite{schmidt1985component}, random walks on hypergraphs \cite{carletti2020random}, \cite{chitra2019random}, and hamiltonian cycles in random hypergraph \cite{frieze2012packing}, \cite{ferber2015closing} \cite{nenadov2019powers} are studied by many researchers. However, to our knowledge, the diameter of random hypergraphs, a natural generalization of Erd\H{o}s-R\'{e}nyi graph, is less explored. In this paper, we study the diameter of random uniform hypergraphs in dense regime.
	
	Throughout $t$ is a natural number with $t\ge 2$. A $t$-uniform random hypergraph is a random unifrom hypergraphs, denoted by $ \mathscr{H}(n,t,p)$, on the set of vertices $ [n] = \{1,2,\ldots ,n \}$ and each possible hyperedge of size $t$ is included independently with probability $p$ \cite{barthelemy2022class}. In other words, $\mathscr{H}(n,t,p)$ is the probability space $\left( \mathscr{H} [n], \mathscr{F_n}, \Prob \right) $, where $\mathscr{H} [n]=\mathscr{H} [n,t]$ denote the set of all $t$-uniform hypergraphs on $[n]$, $\mathscr{F}$ is the $\sigma$-algebra of all subsets of $\mathscr{H} [n]$, and probability $\Prob $ assigned as follows:
	if a hypergraph $\mathcal{H} \in \mathscr {H} (n,t,p) $ consists $m$ hyperedges, then   
	\[
	\Prob \left( \mathcal{H} \right) = p^m \left(1-p \right)^{ { \binom {n}{t} }-m}.
	\]
	Note that the random hypergraph $ \mathscr{H}(n,2,p)$ coincides with the Erd\H{o}s-R\'{e}nyi  $\mathscr{G} (n,p)$ random graph. We study the diameter of $\mathscr{H} (n,t,p)$,
	under  the following assumption. We write, $\mathcal H\in \mathscr{H}(n,t,p)$ to mean that $\mathcal H$ is distributed as $\mathscr{H}(n,t,p)$.
	\begin{assumption} \label{ass main}
	 Let $c $ be a positive constant and  $t, d\ge 2$ be  positive integers. Consider the parameters $n$ and $p$ (depending on $c,t,d,n$) such that 
	\[ 
	\frac{ (t-1)^ {d} N^{d} p^{d}} {n}= \log \left( \frac{n^2}{c} \right), \mbox{ where $N={n-1\choose t-1}$} .
	\]
	\end{assumption}
	\noindent \underline{Remark}: Substituting $t= 2$ and $N=n$ into Assumption \ref{ass main} yields Assumption~\ref{ass bollobas}. Note that, the parameter $p$ depends on $t,c,d,n$, we supress them for ease of writing.
	
	\begin{theorem} \label{main theorem}
		Let $t,c,d,n,p$ be as in Assumption~\ref{ass main}, and  $\mathcal H\in \mathscr{H}(n,t,p)$. Then
		\[ 
		\lim_{n \to \infty} \Prob \left( diam( \mathcal{H}) = d \right) = e^{- \frac{c}{2}} \text{ and } \lim_{n \to \infty} \Prob \left( diam(\mathcal{H}) = d+1 \right) = 1- e^{- \frac{c}{2}}.
		\]
	\end{theorem}	
\noindent Note that  \cref{Bollobas theorem} follows as a corollary of \cref{main theorem}. However, we first present a proof of \cref{Bollobas theorem} followed by the proof of \cref{main theorem} to clarify our new method and the related calculations, thereby enhancing the paper's readability and clarity. 

 This paper focuses on studying the diameter of hypergraphs in the dense regime, that is,  \( np \to \infty \) as \( n \to \infty \).  To our knowledge, the study of the diameter of random hypergraphs in the sparse (thermodynamic) regime, when $np\to \lambda$  (constant) as $n\to \infty$, remains open. However, the diameter of Erd\H os-R\'enyi  graphs $\mathscr{G} (n,p)$ is studied extensively in sparse regime. The limit distribution of the diameter in the sub-critical phase,  when $\lambda < 1$, was derived by {\L}uczak \cite{luczak1998random}. Subsequently, the diameter in the  super-critical phase, that is  $\lambda > 1$, determined by Fernholz and Ramachandran \cite{fernholz2007diameter}, and by Riordan and Wormald  \cite{riordan2010diameter}. In the critical phase, when $\lambda =1$, Nachmias and Peres \cite{nachmias2008critical} has derived the diameter of the largest connected component. The extension of these results to random hypergraphs remains as future work.

This work is partially motivated  by the problem of shotgun assembly of graphs. It means the reconstruction of graphs from neighbourhoods of a given radius $r$. We refer to \cite{gaudio2020shotgun} for a precise definition of shotgun assembly. Mossel and Ross initiated the study of shotgun assembly in \cite{mossel2015shotgun}. Later, the shotgun assembly of unlabeled Erd\H s-R\' enyi graphs is studied extensively in \cite{gaudio2020shotgun}, \cite{Huang}. In \cite{adhikarishotgun}, Adhikari and Chakraborty studied the shotgun assembly of Linial Meshulam complexes. One main problem is how large $r$ should be to ensure that a graph can be identified (up to isomorphism) by its $r$-neighbourhoods. It is easy to see that if $r$ is larger than the graph's diameter, then the graph can be reconstructed from its $r$ neighborhoods. Thus, the diameter gives an upper bound on the neighbourhood's size for reconstructability. Using this fact, an upper bound of $r$ for Erd\H s-R\' enyi graphs is derived in \cite[Theorem 4.2]{mossel2015shotgun}. To our knowledge, the shortgun assembly of hypergraphs has not been studied. The notion of shotgun assembly can easly be extended for hypergraphs. Then \Cref{main theorem} implies that the hypergraphs can be reconstructed from their $(d+1)$-neighborhoods when $p$ and $d$ satisfy Assumption \ref{ass main}. 

The rest of the paper is organized as follows. In the next section, we review key tools, including the Stein-Chen method, Chernoff's bound, coupling, and the Harris-FKG inequality, which are essential for proving our results. We also introduce some notations that will be used throughout the paper. Section \ref{sec:ER} is dedicated for proving \Cref{Bollobas theorem}. The proof of \Cref{main theorem} is given in \Cref{sec:HG}. Additionally, two main auxiliary results \Cref{proposition 1} and \Cref{proposition 2} are proved in \Cref{sec:prop1} and \Cref{sec:prop2} respectively.
	
	\section{Key tools And Notations}
	\subsection{The Stein-Chen Method}: In 1972, Charles Max Stein introduced a  method for Gaussian approximation to the distribution of sums of dependent random variables \cite{stein1972bound}. Later, his student Louis Chen modified the method to obtain a Poisson approximation, which is known as the {\it Stein-Chen method}. This method provides a bound on the total variance distance the difference between a Poisson  and an integer valued  random variables \cite{chen1975poisson}. This method has a wide-range of applications, for example, various birthday-related problems \cite{holst1986birthday}, for calculating the length of longest head run, and to count cycles in random graphs \cite{arratia1990poisson}. 
	
	Generally, the Stein method bounds the distance between probability distributions of the random variables $X$ and $Y$, corresponding to a probability metric $d_{\mathcal{H}}$, defined by
	\[ 
	d_\mathcal{H}(X,Y)= \sup\limits_{h \in \mathcal{H}} \left\vert \E h(X) -\E h(Y) \right\vert, 
	\]
	where $\mathcal{H}$ is a suitable class of functions.
	If the distribution of $Y$ is known, the Stein equation 
	\[
	\mathcal{A}f_h(x)= h(x)-\E h(Y)
	\]
    helps to bound $d_\mathcal{H} (X,Y)$ by finding an appropriate Stein operator $\mathcal{A}f_h$, which satisfies $\E [ \mathcal{A}f_h(Y)]= 0 $.

	Let $\Poi(\lambda)$ denote a Poisson random variable with mean $\lambda$. 
	Let $\mathbb Z^+$ and $\mathbb{R}$ denote the non-negative integers and the real numbers respectively.  We write $Y \sim Poi(\lambda) $ to mean that the random variable $Y$  has same distribution as $Poi(\lambda)$, that is, 
	\[
	\Prob \left( Y = k \right) = \frac{e^{- \lambda} \lambda ^k}{k!},\;\; k=0,1,2, \ldots .
	\]
	The key idea of the Stein-Chen method is inspired from the following characterization of Poisson distribution. A non-negative, integer-valued random variable $Y$ with $Y \sim Poi(\lambda)$, if and only if 
	\[ \lambda \E[f(Y+1)] = \E[Yf(Y)] ,\]
	for all bounded $f:\mathbb{Z}^+ \to \mathbb{R}$. Suppose, for a given function $h:\mathbb{Z}^+ \to \mathbb{R}$, that $f = f_h$ solves the following Stein equation
	\[
	h(x) - \E [ h( Y)] =\lambda f(x+1) - xf(x),
	\]
	where $Y\sim \
	Poi(\lambda)$.
	Replacing $x$ by a random variable $X$ we obtain
	\begin{align}\label{eqn:stein-chen}
	d_\mathcal{H}(X,Y) = \sup\limits_{h \in \mathcal{H}} \left\vert \E h(X) -\E h(Y) \right\vert= \sup\limits_{h \in \mathcal{H}} \left\vert 
	\E[ \lambda f(X+1) - Xf(X)] \right\vert ,
	\end{align}
	the difference between distributions of $X$ and $Y \sim Poi(\lambda)$. In particular, if 
\[
\mathcal{H}=\{ \mathbf{1}_A : A \text{ is a measurable set} \},
\]
 then $d_\mathcal{H}$ is called the total variation distance. Recall, the {\it total variation distance} between two  integer-valued random variables $X$ and $Y$, denoted by $d_{TV} (X,Y) $, is defined as follows
\[ 
d_{TV}(X,Y)= \sup\limits_{A \subseteq \Z} \left\vert \Prob (X \in A)-\Prob (Y\in A) \right\vert .
\]
If $X$ is the sum of Bernoulli random variables then various upper bounds have been derived from \eqref{eqn:stein-chen} in terms of their moments, depending on the dependent structure of the Bernoulli random variables. For more details, see \cite{daly2020stein}. Sometimes the first and second moments  suffice for establishing Poisson convergence using the Stein-Chen method, as discussed in \cite{arratia1989two}.  However, in this paper, we use the size-biased coupling approach to the Stein–Chen method for Poisson approximation when  Bernoulli random variables are  ‘positively dependent’. 

%The size-biased coupling approach to the Stein–Chen method for Poisson approximation is often easiest to apply in the presence of monotonicity, either some form of ‘negative depen- dence’ or ‘positive dependence’.

Given a non-negative random variable $W$ with $\E W > 0$, then the random variable $W^*$ is said to have the {\it $W$-size biased distribution} if
\[
\E [g(W^*)] = \frac { \E [Wg(W)]} {\E W},
\]
for all functions $g: \mathbb{R}^+ \to \mathbb{R}$ for which the above expectations exist.  
\begin{fact}\cite[Lemma 4.13]{daly2020stein} \label{W size biased lemma}
	Suppose $ W = X_1 + \cdots +X_n$ where $X_1,\ldots, X_n$ are Bernoulli random variables with parameters $p_1,\ldots, p_n \in (0,1)$ respectively. Let $J$ be a random variable, independent of all else, with $ \Prob (J = i) = \frac {p_i}{p_1+\cdots+p_n}$. Then
	\[ 
	W^*= 1+ \sum\limits_{i \neq J} X^J _i, \;\;\mbox{  where $X^J _i \overset{\text{d}}{=} \left( X_i \mid X_J = 1 \right)$,}
	\]
	has the W-size biased distribution.
\end{fact}

 We say that a random variable $Y$ is {\it stochastically larger} than another random variable $X$, denoted by $X\le_{\tiny st}Y$, if 
\[
\Prob(X>t)\le \Prob(Y>t), \mbox{ for all } t\in \mathbb R.
\]
This is equivalent to having $\E[g(X)]\le \E[g(Y)]$ for all increasing functions $g$, and to the existence of a coupling $(X',Y')$ of  $(X,Y)$ such that $X'\le Y'$ almost surely. The Bernouli random variables $X_1,\ldots,X_n$ are said to be  {\it positively dependent} if  $ W +1 - X_J \leq _{st} W^* $ holds.  For details see \cite{daly2020stein}, \cite{daly2012stein} and references therein.
\begin{fact}\cite[Theorem 4.14]{daly2020stein}\label{ft:TVbound}
	Let $X_1,\ldots, X_n$,  $p_1,\ldots,p_n$, $ W $ and  $J$ be as defined in Fact~\ref{W size biased lemma}. If $ W +1 - X_J \leq _{st} W^* $ then
	\[
	d_{TV} \left( W,Poi(\E W) \right) \leq \frac{1 - e^{- \E W}}{\E W} \left[ Var(W) - \E W + 2 \sum\limits_{i=1}^{n}  p_i^2\right] .
	\]
\end{fact}

We also use the following fact to prove our results. It provides an upper bound on the total variation distance between two Poisson random variables.
	\begin{fact}[\cite{shang2011isolated}] \label{fact1}
		Let $X$ be a random variable and $\beta \in \mathbb{R}$, then
		\[ d_{TV} \left(Poi(\E[X]),Poi(e^{-\beta }) \right)\leq \left \vert \E [X] -e^{-\beta}\right \vert .\]
	\end{fact}

	\subsection{The Harris-FKG inequality and Chernoff's bound}
	In this subsection, we recall  the Fortuin–Kasteleyn–Ginibre (FKG) inequality and the Chernoff bounds. Here we discuss the FKG inequality in Boolean setup, in this case, it is also known  as the {\it Harris-FKG inequality}. Let $(\Omega, \prec)$ be partially ordered set, where $\Omega=\{0,1\}^n$ and $\prec$ is a partial order on $\Omega$. Suppose $(\Omega, \mathcal F, \Prob)$ is a probability space. We say that an event $A$ is {\it  increasing} if 
	\[
	w\prec w' \mbox{ and } w\in A \mbox{ then } w'\in A.
	\]
	Similarly, an event $A$ is called {\it decreasing} if $w\prec w' \mbox{ and } w'\in A \mbox{ then } w\in A.$
	Then, the Harris-FKG inequality says that events are positively correlated if they are both increasing or both decreasing.

	\begin{fact}[The Harris-FKG inequality]\label{ft:fkg}
		Let $(\Omega, \mathcal F, \Prob)$ be a partially ordered probability space. 	If $A$ and $B$ are both  increasing or both decreasing events then
		\[
		\Prob(A\cap B)\ge \Prob(A)\Prob(B).
		\]
	\end{fact}
	\noindent This type of inequality was first proved by Harris in 1960 \cite{harris1960lower}. However, it is currently known as the FKG inequality, named after Fortuin–Kasteleyn–Ginibre (1971) \cite {fortuin1971correlation} who proved a more general result in the setting of distributive lattices. For more details, see \cite [Section 2.2]{grimmett2013percolation}.

%	We use this fact to show that the random variables $\{X_\alpha\suchthat \alpha \in I\}$ are positively dependent. See the proof Lemma \ref{lem:positivelyrelated}.

Finally, we recall the following well-known theorem, the Chernoff bound, which gives a concentration bound for the sum of independent Bernoulli random variables.

\begin{fact} [Chernoff Bound \cite{mitzenmacher2017probability}, \cite{chernoff1952measure}]
Suppose $ X_1,\ldots, X_n$ are independent Bernoulli random variables and  $S_n = \sum_{i=1}^n X_i$. Then, for $0 < \delta < 1$,
 \[ \Prob \left( \left\vert S_n - \E [S_n] \right\vert
\geq \delta \E [S_n] \right) \leq 2e^{-\delta ^2 \E [S_n] /3 }. 
\]
%\begin{enumerate}[(a)]
%	\item for $0 < \delta \leq 1$,
%	\[ \Prob \left( S_n \geq \left( 1 + \delta \right) \E [S_n] \right) \leq e^{-\delta ^2 \E [S_n] /3 }; 
%	\]
%	\item for $0 < \delta < 1$,
%	\[ \Prob \left( S_n \leq \left( 1 - \delta \right) \E [S_n] \right) \leq e^{-\delta ^2 \E [S_n] /2 }; 
%	\]
%    \item for $0 < \delta < 1$,
%    \[ \Prob \left( \left\vert S_n - \E [S_n] \right\vert
%   \geq \delta \E [S_n] \right) \leq 2e^{-\delta ^2 \E [S_n] /3 }. 
%   \]
%\end{enumerate}
\end{fact}

\subsection{Exchangeable Sequence}
A finite or infinite sequence 
$X_1, X_2, X_3, \ldots$ of random variables is said to be an exchangeable sequence if for any finite permutation $\sigma$ of the indices $1, 2, 3, \ldots$ (the permutation acts on only finitely many indices, 
with the rest fixed), the joint distribution of permuted sequence is same as the original sequence, that is,
\[
\left( X_{\sigma(1)}, X_{\sigma(2)}, X_{\sigma(3)}, \ldots \right)
\overset{d}{=} \left( X_1, X_2, X_3 ,\ldots \right) .
\]

\subsection{Remote pairs} Let $G$ be a graph  on $n$ vertices and  $d$ be a positive integer. We call a pair of vertices $(x,y)$ of $G$ is a {\it remote pair} in $G$ if $d_G (x,y)> d$. Let 
\begin{equation} \label{def I}
	I = \{ (x,y): x,y \in [n] \text{ such that } 1 \leq x <y \leq n \} .
\end{equation}
It is clear that 
$\vert I \vert = \frac{n(n-1)}{2}$, where $|\cdot|$ denotes the cardinality of the set.  For each $ \alpha = (x,y)\in I$ and $G\in \mathcal G(n,p)$, define the Bernoulli random variable $X_ \alpha$ as follows:
\begin{align}\label{eqn:xalpha}
X_ \alpha =
\begin{cases}
	1 & \text{when } d_G(x,y)>d \\
	0 & \text{otherwise.}
\end{cases}
\end{align}
The total number of remote pairs in the graph G is denoted by 
\begin{equation} \label{def W_n}
	W_n = \sum_{\alpha \in I} X_\alpha .
\end{equation}
Observe that the random variable $W_n$ is defined on $\mathcal G(n,p)$. The following proposition is the key result for the proof of \Cref{Bollobas theorem}. We write $ X_n \xrightarrow d X$ to denote that a sequence of random variables $X_n$ converges to $X$ in distribution as $n\to \infty$.
\begin{proposition} \label{proposition 1}
	Let $c,d,n,p$ be as in Assumption~\ref{ass bollobas}, and let $W_n$ be the total number of remote pairs of $G\in \mathcal G(n,p)$. Then, we have
	\[
	W_n \xrightarrow{d} Poi \left( \frac{c}{2} \right),  \;\;\text{ as } n \to \infty.
	\]
	%	where $\xrightarrow D$ denotes the convergence in distribution. 
\end{proposition}
\noindent   In other words, \cref{proposition 1} states that the total number of remote pairs in  $ G \in \mathscr{G} (n,p)$ converges in distribution to a Poisson random variable with mean $c/2$.

Let  $\mathcal H$ be a  $t$-uniform hypegraph on $n$ vertices. Similarly, we call a pair of vertices $(x,y)$  in $\mathcal H \in \mathscr{H} (n,t,p)$ is  remote if $d_\mathcal{H} (x,y)>d$, where $d $ is a positive integer.  With an abuse of notation,  the total number of remote pairs in $\mathcal H$ is denoted by $W_n$. The following proposition, a key result for proving  \Cref{main theorem}, states
that $W_n$ is asymptotically a Poisson  random variable with mean $c/2$.
%Also, we have 
%\[ N = \binom {n-1}{t-1} = \frac{(n-1)(n-2)...(n-t+1)}{(t-1)!} . \]
\begin{proposition} \label{proposition 2}
	Let $c,t, d, n, p$ be as in Assumption~\ref{ass main}, and let $W_n$ be the total number of remote pairs in $\mathcal H\in \mathscr{H} (n,t,p)$. Then, we have
	\[
	W_n \xrightarrow{d} Poi \left( \frac{c}{2} \right), \;\; \text{ as } n \to \infty.
	\]  
\end{proposition}   
\noindent We apply the Stein-Chen method, specifically Fact \ref{ft:TVbound}, to prove \Cref{proposition 1} and \cref{proposition 2}. This proof technique differs from that in \cite{bollobas1981diameter}. Sections \ref{sec:prop1} and \ref{sec:prop2} are dedicated to proving these propositions.

\subsection{Notations}
In this subsection, we define some notations that will be used in the rest of the paper.
Recall $\mathcal{G}[n]$ denotes the set of all graphs on $n$ vertices $[n]:=\{1,\ldots,n\}$.
For $G \in \mathscr{G}[n]$, $x\in [n]$ and $k\in \N$, we denote $\Gamma_k (x)$ and $N_k (x) $ for the set of vertices at distance $k$   and  within distance $k$ from $x$ respectively, that is,
\[
\Gamma_k (x) := \{ y\in V(G) : d_G(x,y)=k \} \quad \text{and} \quad N_k(x):= \cup_{i=0}^{k} \Gamma_i (x).
\]
Let  $d \geq 2$ be a positive integer. Observe that $diam (G)= d$ if and only if $N_d (x)=[n]$ for every vertex $x$ and $N_{d-1} (y) \neq [n] $ for some vertex $y$.
%Let $x,z \in [n]$ and $1 \leq k \leq d-1$. Next, we define $ \Gamma_{k}^* (x,z)$ as follows:
%\begin{align*}
%	\Gamma_{k}^* (x,z) = \{y\in \Gamma_{k}(x) \cap \Gamma_{k}(z): \Gamma (y) \cap & \left( \Gamma_{k-1}(x) - \Gamma_{k-1}(z) \right) \neq \phi \text{ and } \\
%	& \Gamma (y) \cap \left( \Gamma_{k-1}(z) - \Gamma_{k-1}(x) \right) \neq \phi \} .
%\end{align*}

 With abuse of notation, we use the same notation $\Gamma_k (x)$ and  $N_k (x)$   for the hypergraph $\mathcal{H} \in \mathscr{H}[n]$, which are defined in a similar manner.

We also use the notation $f(n)= o(g(n))$ as $n\to \infty$ if  $f(n)/g(n) \to 0$ as $n \to \infty$. %Similarly, $g(n)= o(n)$ means $ \frac {g(n)}{n} \to 0$ as $n \to \infty$. 
 We write $f(n)\approx g(n) $ as $n\to \infty$ if $f(n)/g(n)\to 1$ as $n\to \infty$.
 
 % Throughout, the cardinality of a finite set $A$ is denoted by $|A|$.

\section{Proof of \cref{Bollobas theorem}}\label{sec:ER}
This section, we give the proof of \Cref{Bollobas theorem}. We first complete the proof of \Cref{Bollobas theorem} using \Cref{proposition 1} and the following lemma.

\begin{lemma} \label{probability comparision}
	Let $ 0 \leq p_1 < p_2 \leq 1$, and $r$ be a positive integer, then 
	\[
	\Prob _1 \left( \text{diam}(G)\leq r \right) \leq \Prob _2 \left( \text{diam}(G)\leq r \right),
	\]
	where $\Prob _i $ denotes the probability in the space $ \mathscr{G} (n,p_i ) , 1 \leq i \leq 2$.
\end{lemma}

 For the sake of completeness, we provide a proof of \Cref{probability comparision} at the end of this section. A proof of \Cref{proposition 1} is provided in the next section; this proof technique is different from that used in \cite{bollobas1981diameter}. For completeness, we proceed to prove \Cref{Bollobas theorem} following the methods as in  \cite{bollobas1981diameter}.

\begin{proof}[Proof of \Cref{Bollobas theorem}]
	Observe that, for a positive integer  $d$,  if there is no remote pair of vertices of $G$, then the diameter of $G$ is less than or equal to $d$.
	Therefore \cref{proposition 1} implies that
	\begin{align}\label{eqn:lessd}
	 \Prob \left( diam(G)\leq d \right) = \Prob \left( W_n =0 \right) \to e^{-c/2}, \mbox{ as $n\to \infty$.} 
	\end{align}
Note that if $diam(G)\le 1$ then the $G$ is a complete graph.	Therefore we have
	\[
	 \Prob \left( diam( G)\leq 1\right) =p^{n(n-1)/2} \to 0,\ \mbox{ as } n\to \infty. 
	\]
%Which implies that if  $d=2$, then
%	\[
%	\lim_{n \to \infty} \Prob \left( diam( G) = 2 \right) = e^{- \frac{c}{2}},
%	\]
	Suppose that $d\geq 3$. For given $L_1>0$, choose $0 < p_1 < 1$ such that 
	\[  n^{d-1} p_1^{d} = \log \left( \frac{n^2}{L_1} \right).
	\]
	It is easy to see that  $p < p_1$. Then applying \cref{probability comparision}, we obtain
	\[
	\Prob \left( diam(G)\leq d-1\right) \leq \Prob _1 \left( diam(G)\leq d-1\right) \to e^{-L_1 /2},
	\]
	where $\Prob_1$ denotes the probability in the space $\mathscr{G} \left( n,p_1 \right)$. Since $L_1$ is arbitrary, choosing $L_1 \to \infty$, we get
	\( \Prob \left( diam(G)\leq d-1\right) \to 0 . \)
	Thus, for every $d\geq 2 ,$ we have 
	\begin{equation} \label{eq 1}
		\lim_ {n \to \infty} \Prob \left( diam(G)\leq d-1\right) = 0 . 
	\end{equation} 
	Therefore, combining \eqref{eqn:lessd}  and \eqref{eq 1}, we obtain the first part of the theorem, that is,
	\[
	\lim_{n\to \infty} \Prob \left( diam(G)=d \right)=e^{-\frac{c}{2}}.
	\]
	On the other hand,  for given $L_2>0$, choose $0 < p_2 < 1$ such that 
	\[ n^{d} p_2^{d+1} = \log \left( \frac{n^2}{L_2} \right).
	\]
	Then, $p_2 < p$ and hence applying \cref{probability comparision}, we obtain
	\[
	\Prob_2 \left( diam(G)\leq d+1\right) \leq \Prob \left( diam(G)\leq d+1\right) .
	\]
	Which further implies that 
	\[
	\Prob \left( diam(G) > d+1\right) \leq \Prob_2 \left( diam(G) > d+1 \right) \to 1- e^{-L_2/2} .
	\]
	Choosing $ L_2 \to 0$, we obtain
	\begin{equation} \label{eq 2}
		\lim_ {n \to \infty} \Prob \left( diam(G) > d +1\right) = 0.
	\end{equation}
	We conclude the proof of the theorem using \eqref{eqn:lessd} and \eqref{eq 2}.
\end{proof}

\begin{proof}[Proof of \cref{probability comparision}]
	Let $ E = \left\{ \{i,j\} : 1 \leq i < j \leq n \right\}$ denote the set of all $ \binom {n}{2} $ edges on the vertex set $ [n]$.  Suppose that  $ X= \left( X_{ij} : \{ i,j \} \in E \right)$ and $ Y= \left( Y_{ij} : \{ i,j \}  \in E \right)$ are two random vectors, where $ X_{ij} $ and $ Y_{ij}, 1 \leq i < j \leq n,$ are independent and identically distributed (i.i.d.) Bernoulli($p_1$) and Bernoulli($p_2$) random variables respectively. On the other hand, suppose $ Z= \left( Z_{ij}: \{ i,j \} \in E \right)$
	on $ [0,1]^{ \binom {n}{2}} $, where $ Z_{ij}, 1 \leq i < j \leq n $ are i.i.d. uniform random variables on $ [0,1]$. Define 
	\[
	X^{\prime}= \left( \mathbbm{1}_ { \{ Z_{ij} \leq p_1 \} } : \{i,j\}\in E\right)\;\; \and\;\; Y^{\prime}= \left( \mathbbm{1}_ { \{ Z_{ij} \leq p_2 \} } : \{i,j\}\in E \right).
	\]
	It is clear that $ X^{\prime} \overset{d}{=} X $ and $ Y^{\prime} \overset{d}{=} Y $. Thus $ \left( X^{\prime}, Y^{\prime} \right)$ is a coupling of $X$ and $Y$. Note that, for  $p_1< p_2$, we have 
		\[
	\mathbbm{1}_ { \{ Z_{ij} \leq p_1 \} } \left( w_{ij} \right) \leq  \mathbbm{1}_ { \{ Z_{ij} \leq p_2 \} } \left( w_{ij} \right), \text{ for all }  w_{ij} \in [0,1] .
	\]
	%	%: \left(  [0,1]^{ \binom {n}{2} }, \Prob \right) \to \left( \mathscr{G}[n], \Prob_1 \right) 
%	as follows: for $ w = \left( w_{ij}: 1 \leq i <j \leq n \right) \in [0,1]^{ \binom {n}{2} } $,
%	\begin{align*}
%		X^{\prime} (w) = & \text{a graph } G \text{ with vertex set } [n] \text{ and includes the edge } \{ i,j \} \text{ if } \mathbbm{1}_ { \{ Z_{ij} \leq p_1 \} } (w_{ij}) =1 .
%	\end{align*}
%	Then,
%	\begin{align*}
%		\Prob _1 (G) &= \Prob \circ \left( X^{\prime} \right)^{-1} (G) \\
%		&= \Prob \{ w \in [0,1]^{ \binom {n}{2} } : X^{\prime} (w) = G \} \\
%		&= \Prob \left \{ \left( w_{ij} \right) \in [0,1]^{ \binom {n}{2} } :  \mathbbm{1}_ { \{ Z_{ij} \leq p_1 \} } (w_{ij}) =
%		\begin{cases}
%			1, G \text{ has edge } \{ i,j \} \\
%			0, \text{ otherwise}
%		\end{cases}
%		\right \} \\
%		&= p_1^{ \vert E(G) \vert }
%		\left( 1- p_1 \right) ^{ { \binom {n}{2} } - \vert E(G) \vert } .
%	\end{align*}
%	Consequently, the range space is $ \mathscr{G}(n,p_1)$. Similarly, for $p_2$, we define the random variable $Y^{\prime}= \left( \mathbbm{1}_ { \{ Z_{ij} \leq p_2 \} } : 1 \leq i <j \leq n \right)$ and we obtain the range space $ \mathscr{G}(n,p_2)$. Since $ X^{\prime} \overset{d}{=} X $ and $ Y^{\prime} \overset{d}{=} Y $; $ \left( X^{\prime}, Y^{\prime} \right)$ is a coupling of $X$ and $Y$. Clearly,
%	\[
%	 \mathbbm{1}_ { \{ Z_{ij} \leq p_1 \} } \left( w_{ij} \right) \leq  \mathbbm{1}_ { \{ Z_{ij} \leq p_2 \} } \left( w_{ij} \right), \text{ for all }  w_{ij} \in [0,1] .
%	\]
Observe that, for each $w\in [0,1]^{\binom{n}{2}}$, we can identify $X'(w)$ with a graph in $\mathcal G[n]$.	Thus $X^{\prime} (w)$ is a subgraph of $Y^{\prime} (w)$ for each $w\in [0,1]^{\binom{n}{2}}$. Therefore, for $(x,y)\in I$,
	\begin{align*}
		\Prob_1 \left( \{ G \in \mathscr{G}[n] :d_G (x,y) \leq r \} \right) 
		%&= \Prob \circ (X^{\prime} )^{-1}\left( \{ G \in \mathscr{G}[n] :d_G (x,y) \leq r \} \right) \\
		&= \Prob \left( \{  w \in [0,1]^{ \binom {n}{2} } :X^{\prime}(w )= G \text{ s.t. } d_G (x,y) \leq r \} \right) \\
		&\leq \Prob \left( \{  w \in [0,1]^{ \binom {n}{2} } :Y^{\prime}(w )= G \text{ s.t. } d_G (x,y) \leq r \} \right) \\
		&= \Prob \circ (Y^{\prime} )^{-1}\left( \{ G \in \mathscr{G}[n] :d_G (x,y) \leq r \} \right) \\
		&= \Prob_2 \left( \{ G \in \mathscr{G}[n] :d_G (x,y) \leq r \} \right).
	\end{align*}
	Thus,  $\Prob_1 \left( d_G (x,y) \leq r, \text{ for all } (x,y) \in I \right) \leq \Prob_2 \left( d_G (x,y) \leq r, \text{ for all } (x,y) \in I \right)$. Hence
	\[
	 \Prob _1 \left( \text{diam}(G)\leq r \right) \leq \Prob _2 \left( \text{diam}(G)\leq r \right).
	\]
	This completes the proof.
\end{proof}

%%%%%%%%%%%%%%%%%%%%%%%%%%%%%%%%%%%%%%%%%%%%%%

\section{Proof of \Cref{proposition 1}}\label{sec:prop1}
In this section, we provide the proof of \Cref{proposition 1}, which explains the key techniques and calculations used in the paper and will be further utilized in the hypergraph setup in \Cref{sec:HG}. This will enhance the paper's readability and clarity. The following lemmas will be used in the proof of \Cref{proposition 1}.

\begin{lemma} \label{X alpha=1}
	Let $c,d,n,p$  be as in Assumption \ref{ass bollobas}, and $G \in \mathscr{G}(n,p)$.  Suppose $I$ is the index set as defined in \eqref{def I}. Then, each $\alpha\in I$,
\begin{align*}
\Prob \left( X_ \alpha =1 \right) \approx \frac{c}{n^2}, \mbox{ as } n\to \infty,
\end{align*}	
where $X_\alpha$ is as defined in \eqref{eqn:xalpha}.
% If for each $\alpha = (x,y) \in I,$
%	\[
%	X_ \alpha =
%	\begin{cases}
%		1 & \text{when} \quad d_G (x,y)>d \\
%		0 & \text{otherwise}
%	\end{cases},
%%	\]
%	then \[ \left( \frac{c}{n^2} \right)^ {(1+p) \left( 1+ \eta_{d-1} \right)} \left( 1- o(1) \right) \leq \Prob \left( X_ \alpha =1 \right) \leq \left( \frac{c}{n^2} \right)^{1- \eta_{d-1}} \left( 1+ o(1) \right). \]
\end{lemma}

\begin{lemma} \label{X alpha and X beta=1}
		Let $c,d,n,p$  be as in Assumption \ref{ass bollobas}, and $G \in \mathscr{G}(n,p)$.  Suppose $I$ is the index set as defined in \eqref{def I}. Then, for $\alpha,\beta \in I$ and $\alpha\neq \beta$,
	\begin{align*}
		\Prob \left( X_ \alpha =1,X_\beta = 1 \right) \approx \frac{c^2}{n^4}, \mbox{ as } n\to \infty,
	\end{align*}	
	where $X_\alpha$ is as defined in \eqref{eqn:xalpha}.
%	If $ \alpha , \beta \in I$, then 
%	\[
%	\left( \frac{c^2}{n^4} \right)^{(1+p)(1+ \eta_{d-1} ) \left( 1-o(1) \right)}
%	\left( 1-o(1) \right)
%	\leq \Prob \left( X_ \alpha =1,X_\beta = 1 \right) \leq  \left( \frac{c^2}{n^4} \right) ^{ \left( 1 - \eta_{d-1} \right) \left( 1-o(1) \right) } \left( 1+o(1) \right).
%	\]
\end{lemma}

\begin{lemma}\label{lem:positivelyrelated}
	Let $\{X_\alpha\suchthat \alpha\in I\}$ be the Bernoulli random variables as defined in \eqref{eqn:xalpha}. Suppose $J, W$ and $W^*$ be as defined in Fact \ref{ft:TVbound} then 
	\[
	W+1-X_J\le_{\tiny st} W^*.
	\]
\end{lemma}
\begin{proof}[Proof of \Cref{lem:positivelyrelated}]
Recall, the probability space $\mathcal G(n,p)=(\mathcal G[n],\mathcal F,\mathbb P)$ where the random variables $\{X_\alpha\}_{\alpha\in I}$ are defined.  Observe that the random variables $\{X_\alpha\}_{\alpha\in I}$ are exchangeable. Thus it is enough show that,  for fixed $\alpha\in I$, 
\begin{align}\label{eqn:alphast}
\sum_{\beta\neq \alpha}X_\beta \le_{\tiny st}\sum_{\beta\neq \alpha}X_\beta^\alpha.
\end{align}
Indeed, as $J$ is independent of $\{X_\alpha\}_{\alpha\in I}$, for $t \in \mathbb{R}$, from \eqref{eqn:alphast} we have
\begin{align*}
\Prob(W+1-X_J>t)&=\sum_{\alpha\in I}\Prob(W+1-X_J>t\given J=\alpha)\Prob(J=\alpha)
\\&=\sum_{\alpha\in I}\Prob(W+1-X_\alpha>t)\Prob(J=\alpha)
\\&\le \sum_{\alpha\in I}\Prob(1+\sum_{\beta\neq \alpha}X_\beta^\alpha>t)\Prob(J=\alpha)
\\&=\sum_{\alpha\in I}\Prob(1+\sum_{\beta\neq \alpha}X_\beta^\alpha>t\given J=\alpha)\Prob(J=\alpha)
\\&\le \Prob(W^*>t).
\end{align*}
Hence the result, that is, $W+1-X_J\le_{\tiny st} W^*$. It remains to prove \eqref{eqn:alphast}. We first show that $X_\beta\le_{\tiny st} X_\beta^\alpha$. We have $X_\beta^{\alpha}\stackrel{d}{=}(X_\beta\given X_\alpha=1)$. In particular, we have
\(
\Prob(X_\beta^\alpha=1)=\Prob(X_\beta=1\given X_\alpha=1).
\)
Thus it is equivalent to show that 
\[
\Prob(X_\beta=1, X_\alpha=1)\ge \Prob(X_\beta=1)\Prob(X_\alpha=1).
\]
Observe that $\mathcal G(n,p)$ can be viewed as $(\Omega, \mathcal F, \Prob)$, where $\Omega=\{0,1\}^{\binom{n}{2}}$ and $\Prob$ is the product measure on it. We define the following partial order on $\Omega $ defined as
\[
w\prec w' \mbox{ if } w(u)\ge w'(u),  \mbox{ for all } u\in \l[\binom{n}{2}\r].
\]
Then it is easy to see that $\{X_\alpha=1\}$ is an increasing event, for each $\alpha\in I$.  Therefore by the Harris-FKG inequality, Fact \ref{ft:fkg}, we have
\[
\Prob(X_\beta=1, X_\alpha=1)\ge \Prob(X_\beta=1)\Prob(X_\alpha=1).
\]
Which implies that $X_\beta\le_{\tiny st} X_\beta^\alpha$. Therefore there exists a coupling $(\tilde X_\beta,\tilde X_\beta^{\alpha})$ of $(X_\beta,X_\beta^\alpha)$ such that 
\begin{align}\label{eqn:lessthan}
\tilde X_\beta\le \tilde X_\beta^{\alpha}, \mbox{ almost surely}.
\end{align}
% Observe that, to show \eqref{eqn:alphast}, we need to show that, for $t\ge 0$,
%\begin{align*}
%	\Prob(\sum_{\beta \neq \alpha}X_\beta>t)\le \Prob(\sum_{\beta \neq \alpha}X_\beta^\alpha>t).
%\end{align*} 
With abuse of notation we use $\Prob$ in every step. By \eqref{eqn:lessthan}, for $t\in \R$, we have 
\begin{align*}
	\Prob(\sum_{\beta \neq \alpha}\tilde X_\beta>t)\le \Prob(\sum_{\beta \neq \alpha}\tilde X_\beta^{\alpha}>t).
\end{align*}
Since $X_\beta$ has the same distribution as $\tilde X_\beta$ and $\tilde X_\beta^{\alpha}$ has same distribution as $X_\beta^{\alpha}$, by the last inequality we have 
\[
\Prob(\sum_{\beta \neq \alpha}X_\beta>t)\le \Prob(\sum_{\beta \neq \alpha}X_\beta^\alpha>t), \mbox{ for $t\in \R$}.
\]
This completes the proof of \eqref{eqn:alphast}. Hence the result.
%For fix $\alpha\in I$, define 
%$$
%\mathcal G_\alpha[n]=\{G\in \mathcal G[n]\suchthat X_\alpha(G)=1\}.
%$$
%
%
%Consider a new probability space $(\mathcal G_\alpha[n], \mathcal P(\mathcal G_\alpha[n]), \mathbb P^*)$, where 
%\[
%\mathbb P^*(G)=\frac{\mathbb P(G)}{\mathbb P(\mathcal G_\alpha[n])}, \mbox{ for all } G\in \mathcal G_\alpha[n].
%\]
%For $G\in \mathcal G_\alpha[n]$, we define 
%\[
%Y_{\beta}(G)=\left\{\begin{array}{ll}
%	1 &\mbox{ if } d_G(w,z)>d
%	\\0 & \mbox{ otherwise}.
%\end{array}
%\right.
%\]
%Obeserve that $Y_\beta$ has the same law as $\{X_\beta\given X_\alpha\}$. Indeed, for $\beta=(w,z)$, we have
%\begin{align*}
%	\mathbb P^*(\{Y_\beta=1\})&=\mathbb P^*(\{G\in \mathcal G_{\alpha}[n]\suchthat d_G(w,z)>d\})
%	\\&=\frac{\mathbb P(\{G\in \mathcal G_{\alpha}[n]\suchthat d_G(w,z)>d\})}{\mathbb P(\mathcal G_\alpha[n])}
%	\\&=\frac{\mathbb P(\{G\in \mathcal G[n]\suchthat d_G(w,z)>d, d_G(x,y)>d\})}{\mathbb P(\{X_\alpha=1\})}
%	\\&=\mathbb P(\{X_\beta=1\given X_\alpha=1\}).
%\end{align*}
%For $\beta\neq \alpha$ and $G\in \mathcal G[n]$, we define 
%\[
%X_{\beta \alpha}(G)=Y_\beta(G'),\mbox{ where $G'\in \mathcal G_\alpha[n]$ and $G\in [G']$}.
%\]
%Obeserve that, $\mathcal L(X_{\beta\alpha})=\mathcal L(Y_\beta)$ for all $\beta\neq \alpha$.  It is clear that, by the construction, $X_{\beta \alpha}(G)\ge X_\beta(G')$ for all $G'\in \mathcal G_\alpha[n]$ and $G\in [G']$.
\end{proof}

 Now we proceed to prove Proposition \ref{proposition 1}. The proofs of \Cref{X alpha=1} and \Cref{X alpha and X beta=1} will be given in the next two subsections.

  \begin{proof}[Proof of \Cref{proposition 1}]
  	Recall $W_n $ as defined in \eqref{def W_n}, the total number of remote pairs. To prove the result, we show that
  	\[
  	d_{TV} (W_n,Poi(c/2)) \to 0 \,\, \mbox{ as } n\to \infty.
  	\] 
  	By the triangle inequality, we have the following inequality
  	\begin{equation} \label{eq 3}
  		d_{TV} \left(W_n,Poi( c/2 )\right) \leq d_{TV} \left(W_n,Poi(\E W_n)\right) + d_{TV} \left(Poi(\E W_n), Poi(c/2)\right).
  	\end{equation}
  	First we estimate the second term of the right hand side of \eqref{eq 3}.  \cref{X alpha=1} implies 
  	\[
  	\E [W_n]=|I|\Prob \left( X_ \alpha =1 \right) \approx \frac{c}{2},\;\; \mbox{ as } n\to \infty.
  	\]
  	%  	\[ \vert I \vert \left[ (\frac{c}{n^2})^ { (1+ \eta_{d-1}) (p+1)} ( 1 - o(1))  \right] \leq \E [W_n] \leq \vert I \vert \left[ (\frac{c}{n^2})^ { (1 - \eta_{d-1})} ( 1 + o(1)) \right]\]
  	Applying the last equation and Fact \ref{fact1}, we get
  	\begin{align}\label{eqn:2ndterm}
  		d_{TV} \left(Poi(\E [W_n]), Poi(c/2)\right) &\leq \left \vert \E [W_n] - c/2 \right \vert \to 0 \text{ as } n \to \infty.
  		%& \leq  \frac{n(n-1)}{2}\left[ (\frac{c}{n^2})^ { (1 - \eta_{d-1})} ( 1 + o(1)) \right] -\frac{c}{2} 
  	\end{align}
  	To estimate the first term in the right hand side of \eqref{eq 3}, we use \Cref{X alpha and X beta=1}, \Cref{lem:positivelyrelated} and Fact \ref{ft:TVbound}. 
%  	To do so, we define the random variables $(X_{\beta \alpha})_{\beta \in I\backslash \{\alpha\}}$ for each $\alpha = (x,y) \in I$ as follows
%  	\[ X_{\beta \alpha}=
%  	\begin{cases}
%  		1& \text{when } \beta \text{ is remote pair in } G^ \prime \\
%  		0& \text{otherwise}
%  	\end{cases} \]
%  	where the graph $ G^\prime$ is obtained by deleting all the paths of length at most $d$ from $x$ to $y$ in the graph $G$, that is, if $d_G (x,y) \leq d$. Observe that $G'=G$ if $d_G (x,y) > d$.
%%  	By coupling of $G$ and $ G^\prime$, for each $\alpha \in I$, we find that
%%  	\[ \{ X_{\beta \alpha}=1\} = \{  X_\beta =1, X_\alpha =1 \} \text{ and } \{ X_{\beta \alpha}=0\} = \{  X_\beta =0, X_\alpha =1\} ,\forall \beta \in I \backslash \{ \alpha \}.
%%  	\]
%Thus if $X_\alpha=1$ then $X_{\beta\alpha}=X_\beta$ for all $\beta\in I\backslash \alpha$.
%  In particular, we have
%  	\[ \mathscr{L} \left( (X_{\beta \alpha})_{\beta \in I \backslash \{\alpha\}} \right) = \mathscr{L} \left( (X_{\beta})_{\beta \in I \backslash \{\beta\} } | X_\alpha =1 \right). \]
%  On the other hand, observe that  if a pair of vertices is remote in $G$ then it will remain remote in $G'$. Thus $X_\beta = 1$ implies that $X_{\beta\alpha} =1$. Therefore
% $$
% X_{\beta \alpha} \geq X_\beta \mbox{ for all } \beta\neq \alpha
% .$$
   % Hence  the random indicator variables $(X_\alpha)_{\alpha \in I}$ are positively related. 
Using \Cref{lem:positivelyrelated} and the inequality $1 - e^{-x} \leq x$ for $x \geq 0$ in Fact \ref{ft:TVbound}, we obtain
  	\begin{equation} \label{eq 4}
  		d_{TV} (W_n,Poi(\E W_n)) 
  		\leq \E (W_n^2) -(\E W_n)^2-\E W_n + 2 \vert I \vert (\E X_\alpha)^2 .
  	\end{equation}
  Recall $W_n=\sum_IX_\alpha$. Then	we have 
  	\begin{align*}
  		\E (W_n^2)&= \sum_{\alpha \in I} \E (X_\alpha^2) + \sum_{\alpha \neq \beta}\E \left(  X_\alpha X_\beta \right)\\ &= |I|{\E (X_\alpha)} + \vert I \vert (\vert I \vert -1) \Prob \left( X_ \alpha =1,X_\beta = 1 \right).
  	%	& \leq \E [W_n] + (\vert I \vert^2 - \vert I \vert) \left[  \left( \frac{c^2}{n^4} \right) ^{ \left( 1 - \eta_{d-1} \right) \left( 1-o(1) \right) } \left( 1+o(1) \right)  \right]
  	\end{align*}
  	By \Cref{X alpha=1} and \Cref{X alpha and X beta=1} we get 
  	\[
  		\E (W_n^2)\approx \frac{c}{2} +\frac{c^2}{4}, \mbox{ as } n\to \infty.
  	\]
  	Putting the values of $\E W_n$, $ \E (W_n ^2)$ and $ \E (X_\alpha) $ in the \eqref{eq 4} we get,
  	\begin{align}\label{eqn:1stterm}
  		d_{TV} \left( W_n,Poi(\E W_n) \right) 
  		%&\leq \left( \vert I \vert^2 - \vert I \vert \right) \left[ \left( \frac{c^2}{n^4} \right) ^{ \left( 1 - \eta_{d-1} \right) \left( 1-o(1) \right) } \left( 1+o(1) \right) \right] - \\
  	%	& \vert I \vert ^2 \left[ (\frac{c}{n^2})^ { (1+ \eta_{d-1}) (p+1)} ( 1 - o(1)) \right]^2 + 2 \vert I \vert \left[ (\frac{c}{n^2})^ { (1 - \eta_{d-1})} \left( 1 + o(1) \right) \right]^2 \\
  		& \to 0 \text{ as } n \to \infty .
  	\end{align}
  	Therefore using \eqref{eqn:2ndterm} and \eqref{eqn:1stterm} in \eqref{eq 3} we get the result.
  \end{proof}

%%%%%%%%%%%%%%%%%%%%%%%%%%%%%%%%%%%%%%%%%%%%%%%%%%%%%%%%%%%%%%
\subsection{Proof of \cref{X alpha=1}} In this subsection we provide the proof of \cref{X alpha=1}. We present the proof of \cref{X alpha=1} after proving some auxiliary lemmas.
   For a fixed vertex $x$ and $1 \leq k \leq {d-1}$, suppose $E_k(x)$ denotes the set of edges connecting the vertices of $\Gamma_{k-1} (x)$ to the vertices of $ N_{k-1}^c(x)$, that is,
\[
   E_k(x) = \{e \in  E(G): e \cap \Gamma_{k-1} (x) \neq \emptyset, e \cap N_{k-1}^c (x) \neq \emptyset \} .
\] 
Let $\Omega _{k,x} \subseteq \mathscr{G} [n]$ be the set of graphs for which $ \vert \Gamma _ {k-1} (x) \vert$ and $\vert N_ {k-1} (x) \vert $ satisfy 
\begin{equation} \label{eqn:gammak}
	\frac{1}{2} (pn)^ {k-1} \leq \vert \Gamma _ {k-1} (x) \vert \leq \frac{3}{2} (pn)^ {k-1} \quad \text{and} \quad \vert N_ {k-1} (x) \vert \leq 2 (pn)^ {k-1}.
\end{equation}
In other words, we denote
\begin{equation} \label{def:Omega k,x}
\Omega _{k,x} = \left \{ G \in \mathscr{G} [n]: G \text{ satisfies } \eqref{eqn:gammak}  \right \}.
\end{equation}
The next lemma gives an estimate on $|  E_k(x)|$, for $x\in V$ and $1\le k\le d-1$. 
\begin{lemma} \label{E_k(x)} 
   Let x be a fixed vertex and  $L \geq 72 $ be a constant.
   Define
 \begin{align}\label{eqn:deltak}
   \delta_k = \left [ \frac { L \log n}{ n^k p^k  } \right ]^{1/2}, \mbox{ for $1 \leq k \leq {d-1}$.}
 \end{align}
  Then, for $1 \leq k \leq {d-1}$, for large  $n$ the following holds  
   \[
   \Prob(\{||E_k(x)|-\E[|E_k(x)|]|>\delta_k\E[|E_k(x)|]\}\given \Omega_{k,x})\le 2 n^{-6}.
   \]
%    \[ \left\vert \vert E_k (x) \vert - \vert \Gamma_{k-1}(x) \vert \left( n- \vert N_{k-1}(x) \vert \right) p \right\vert \leq \delta_k \vert \Gamma_{k-1}(x) \vert \left( n- \vert N_{k-1}(x) \vert \right) p, \]
%    with probability at least $1- 2 n^{-L/12}$.
\end{lemma} 
\begin{proof}[Proof of \cref{E_k(x)}]
    Observe that,  given $\vert \Gamma_{k-1}(x) \vert $ and $\vert N_{k-1}(x) \vert$, the random variable $\vert E_k(x) \vert$ has binomial distribution with parameters   
    $\vert \Gamma_{k-1}(x) \vert \left( n- \vert N_{k-1}(x) \vert \right)$ and $p$. Hence,
    \[
    \E \vert E_k(x) \vert = \vert \Gamma_{k-1}(x) \vert \left( n- \vert N_{k-1}(x) \vert \right) p.
    \]
    Applying Chernoff bound for $ \delta_k$ as defined in \eqref{eqn:deltak}, we obtain
    \begin{align}\label{eqn:cher1}
    \Prob \left(\l||E_k(x)|-\E[|E_k(x)|]\r|>\delta_k \E[|E_k(x)|]\right)\le 2e^{-\frac{\delta_k^2}{3}\E[|E_k(x)|]}.
    \end{align}
    Observe that,  given $ \Omega_{k,x}$ and $1\le k\le d-1$, we have
    \[
    n- \vert N_{k-1}(x) \vert = n \left( 1 - o(1) \right), \text {  as  } \vert N_{k-1}(x) \vert =o(n), \mbox{ as } n\to \infty.
    \]
    Thus by \eqref{eqn:gammak} and \eqref{eqn:deltak},  for large $n$, we get
    \[
    \frac{\delta_k^2}{3}\E[|E_k(x)|]\ge \frac{L\log n}{12}.
    \]
    Therefore, using the last equation in \eqref{eqn:cher1}, we get 
    \begin{align*}
    	 \Prob \left(\l||E_k(x)|-\E[|E_k(x)|]\r|>\delta_k \E[|E_k(x)|]\given \Omega_{ k,x}\right)\le 2 n^{-L/12}.
    \end{align*}
    Hence, we obtain the result.
\end{proof}
 
% \noindent  The next result estimates $|E_k (x,v)|$. We say $E_k(x,v)$ and $E_k(x,v')$ are disjoint if $e\cap e'=\emptyset$ for all $e\in E_k(x,v)$ and $e'\in E_k(x,v')$.
\begin{lemma} \label{lemma 4} Let $\delta_1$ be as defined in \eqref{eqn:deltak}. 
	Let  $ 1 \leq k \leq d-1 $ and 
	\begin{align}\label{eqn:epsilonk}
	\epsilon_k = 2\delta_1
	%\delta_1\l(1-\frac{\vert N_{k-1}(x) \vert}{n}\r)+ \frac{\vert N_{k-1}(x) \vert}{n} +\frac{10}{\vert \Gamma_{k-1} (x) \vert np}.
	\end{align}
	Then, for a fixed vertex $x$ and $ 1 \leq k \leq d-1 $, we have
	\[
	\Prob \left(\l||\Gamma_k(x)|-\E[|\Gamma_k(x)|]\r|>\delta_1 \E[|\Gamma_k(x)|]\given \Omega_{ k,x}\right)\le 3n^{-10}.
	\]
\end{lemma}
\begin{proof} [Proof of \cref{lemma 4}]
	Fix $1\le k\le d-1$. For $ v \in  \Gamma_{k-1} (x)$, suppose $ E_k (x,v) $ denotes the set of edges of $E_k (x)$ containing  $v$, that is,
	\[
	E_k (x,v) =  \{ e \in E_k (x) : v\in e \} .
	\]
%	In particular, $ \vert E_1 (x,x) \vert = \deg (x) $.
   Observe that, given $\vert N_{k-1}(x) \vert$, the random variable $ \vert  E_k (x,v) \vert$ has Binomial distribution with parameters $ n- \vert N_{k-1}(x) \vert $ and $p$. Hence, by  Chernoff bound for $\delta_1$, we obtain
   \[
   \Prob \left(\l||E_k(x,v)|-\E[|E_k(x,v)|]\r|>\delta_1\E[|E_k(x,v)|]\right)\le 2e^{-\frac{\delta_1^2}{3}\E[|E_k(x,v)|]}.
   \]
   where $\E[|E_k(x,v)|]= \left(n-|N_{k-1}(x)| \right) p$. By \eqref{eqn:gammak},  \eqref{eqn:deltak}, and by following similar steps as in \cref{E_k(x)}, for large $n$, we get
   \[
    \Prob \left(\big||E_k(x,v)|-\E[|E_k(x,v)|]\big|>\delta_1\E[|E_k(x,v)|] \given \Omega_{ k,x}\right)\le 2n^{-\frac{L}{6}}.
   \]
%   By the union bound and \eqref{eqn:gammak}, given $\Omega_{k,x}$, we have
%   \[
%   \Prob(\{\exists v\in \Gamma_{k-1}(x)\suchthat \big||E_k(x,v)|-\E[|E_k(x,v)|]\big|>\delta_1\E[|E_k(x,v)|]\})\le 2 \vert \Gamma_{k-1} (x)  \vert  n^{-\frac{L}{6}}.
%   \]
    Let $A_v=\left\{\l||E_k(x,v)|-\E[|E_k(x,v)|]\r|\le \delta_1\E[|E_k(x,v)|]\right\}$ and
     \(
    \mathcal{A}= \cap_{v\in \Gamma_{k-1}(x)}A_v.
    \) Then, given $\Omega_{ k,x}$, by the union bound we have
     \begin{align} \label{cher:mathcalA}
     \Prob(\mathcal A)\ge 1- \frac{ 2 \vert \Gamma_{k-1} (x)  \vert }{ n^{L/6}}.
     \end{align}
Since $\Gamma_k(x)=\cup_{v\in \Gamma_{k-1}(x)}E_k(x,v)$ and $\E[|E_k(x,v)|]=\left(n-|N_{k-1}(x)| \right) p$, therefore 
	\begin{align*}
		\Prob(|\Gamma_{k}(x)|>(1+\delta_1)|\Gamma_{k-1}(x)|(n-|N_{k-1}(x)|)p)\le  \frac{ 2 \vert \Gamma_{k-1} (x)  \vert }{ n^{L/6}}.
	\end{align*}
Since $L\ge 72$ and $|\Gamma_{k-1}(x)|=o(n)$, for $\epsilon_k$ as defined in \eqref{eqn:epsilonk} and large $n$, we have 
	\begin{align}\label{eqn:upperepsilon}
	\Prob\l(|\Gamma_{k}(x)|>(1+\epsilon_k)np|\Gamma_{k-1}(x)|\r)\le  \frac{1}{ n^{10}}.
\end{align}
Next we give lower bound of $|\Gamma_{k-1}(x)|$.
Let $a=(1-\delta_1)(n-|N_{k-1}(x)|)p$ and  $b=(1+\delta_1)(n-|N_{k-1}(x)|)p$. Suppose $\ell=|\Gamma_{k-1}(x)|$ and ${\bf m}=(m_1,\ldots, m_{\ell})$ where $m_1,\ldots, m_\ell\in \N$. We define
\[
\mathcal{A}_{\bf m} = \{ G \in \mathscr{G} [n]: \vert E_k (x,v_i) \vert =m_i  \text{ where } v_i \in \Gamma_{k-1}(x) \}
.\] 
By the definition, the events $\{\mathcal{A}_{\bf m}\}$ are disjoint. Thus $\mathcal{A}$ is the disjoint union of   $\mathcal{A}_m $, where ${\bf m}\in [a,b]^\ell$. 
 Which implies that 
 \begin{align}\label{eqn:sumam}
 \Prob (\mathcal{A})=\sum_{\bf m\in [a,b]^\ell}\Prob(\mathcal{A}_{\bf m}).
 \end{align}
 Let $\partial(\Gamma_{k-1}(x))=\cup_{v\in \Gamma_{k-1}(x)} E_k(x,v)\backslash \Gamma_{k-1}(x)$ denote the boundary of $\Gamma_{k-1}(x)$. We define $\mathcal{A}_{\bf m}^0 $, and $ B_{\bf m}$ as follows, for ${\m}=(m_1,\ldots, m_\ell)$,
   \begin{align*}
  % \vert E_k^{0} (x) \vert &= \sup \{l: E_k (x,v_1), \ldots,  E_k(x,v_l ) \text { are disjoint, } v_1,\ldots,v_l \in \Gamma_{k-1}(x) \}
   \mathcal{A}_{\bf m}^0 &= \{ G \in \mathcal{A}_{\m}: \vert \partial(\Gamma_{k-1}(x)) \vert \geq m_1+\cdots+m_\ell - 10 \}
   \\ B_{\m} &= \{ G \in \mathcal{A}_{\m}: \vert \partial(\Gamma_{k-1}(x)) \vert = m_1+\cdots+m_\ell \}.
   \end{align*}
   Observe that $B_{\m}$ occurs when  the sets $E_k(x,v_1),\ldots,E_k(x,v_\ell)$ are disjoint. We say $E_k(x,v)$ and $E_k(x,v')$ are disjoint if $e\cap e'=\emptyset$ for all $e\in E_k(x,v)$ and $e'\in E_k(x,v')$.
It is clear that $\mathcal A_{\m}^0\subset \mathcal A_{\m}$ and  $\Prob(\mathcal A_{\m})=\Prob(A_{\m}^0)+\Prob(\mathcal A_{\m}\cap (\mathcal A_{\m}^0)^c)$. Which implies
\begin{align*}
    \Prob \left( \mathcal{A}_{\m}^0 \right)
    % &= \Prob \left( \mathcal{A}_m \right) \left [1 - \frac { \Prob \left( \mathcal{A}_m \cap (A_m^0)^c \right)} {\Prob \left( \mathcal{A}_m \right) } \right] \\
    &= \Prob \left( \mathcal{A}_{\m} \right) \left [1 - \frac {\vert \mathcal{A}_{\m} \cap (\mathcal A_{\m}^0)^c \vert} {\vert \mathcal{A}_{\m} \vert} \right] 
    \geq \Prob \left( \mathcal{A}_{\m} \right) \left [1 - \frac {\vert \mathcal{A}_{\m} \cap (A_{\m}^0)^c \vert} {\vert B_{\m} \vert} \right].
%     \\
%    &= \Prob \left( \mathcal{A}_m \right) \left [1 - \frac { \left( n- \vert N_{k-1}(x) \vert -m \vert \Gamma_{k-1}(x) \vert \right)! (1+o(1))}
%    { \left( n- \vert N_{k-1}(x) \vert -m \vert \Gamma_{k-1}(x) \vert +11 \right)! }
%    \right] \\
%    &\geq \Prob \left( \mathcal{A}_m \right) \left [1 - \frac { (1+o(1))}
%    { \left( n- \vert N_{k-1}(x) \vert -m \vert \Gamma_{k-1}(x) \vert \right)^ {11} }
%    \right],
\end{align*}
The last inequality follows from the fact that $B_{\m}\subset \mathcal A_{\m}$. Observe that, we have
\begin{align*}
\vert B_{\m} \vert &= \left( n- \vert N_{k-1}(x) \vert \right) \cdots  (n- \vert N_{k-1}(x) \vert -(m_1+\cdots+ m_\ell-1)),
%\\
%&= \frac{ \left( n- \vert N_{k-1}(x) \vert \right) ! }{ \left( n- \vert N_{k-1}(x) \vert -m \vert \Gamma_{k-1}(x) \vert \right)! },
\end{align*}
as all the end points (which are not in $\Gamma_{k-1}(x)$) of edges are distinct. On the other hand, we have 
\begin{align*}
\vert \mathcal{A}_{\m} \cap (A_{\m}^0)^c \vert \le
 \left(n- \vert N_{k-1}(x) \vert \right) \cdots (n- \vert N_{k-1}(x) \vert -( m_1+\cdots+m_\ell - 12)),
% \left( 1+o(1) \right) \\
%&= \frac{\left( n- \vert N_{k-1}(x) \vert \right) !}{\left( n- \vert N_{k-1}(x) \vert -m \vert \Gamma_{k-1}(x) \vert +11 \right)!}.
\end{align*}
as there are at least  $11$ edges repeated, that is, the end points of at most $(n- \vert N_{k-1}(x) \vert -( m_1+\cdots+m_\ell - 12)$ edges are distinct. Therefore by \eqref{eqn:gammak} we get
\begin{align}\label{eqn:am0lower}
 \Prob \left( \mathcal{A}_{\m}^0 \right)\ge  \Prob ( \mathcal{A}_m ) \left(1 - \frac {2 } { n^{11} }  \right), \hspace{.3cm}\mbox{ for large $n$.}
\end{align}
 Suppose that $\mathcal A^0=\{G\in \mathcal G[n] \suchthat |\partial(\Gamma_{k-1}(x))|\ge \ell a-10\}$. Then 
\[ \mathcal{A}^0\supset \bigcup_{\m\in [a,b]^\ell} \mathcal{A}_{\m}^0.
\]
Therefore  by \eqref{eqn:sumam} and \eqref{eqn:am0lower} we have 
\begin{align*}
 \Prob \left( \mathcal{A}^0 \right) \ge  \sum\limits_{ \m\in[a,b]^\ell }\Prob(\mathcal{A}_{\m}) \left(1 - \frac {1 } { n^{10} }  \right)=\Prob(\mathcal A)\left(1 - \frac {2 } { n^{11} }  \right).
\end{align*}
Since  $|\Gamma_{k-1}(x)|=o(n)$ (from \eqref{eqn:gammak} ) and $L\ge 72$, hence  using \eqref{eqn:cher1} we get
\[
\Prob(\mathcal A^0)\ge 1 - \frac {2 } { n^{10} } .
\]
Observe that $\Gamma_k(x)=\partial(\Gamma_{k-1}(x))$, putting values of $\ell$
and 
$a$,  we get 
\[
\Prob(|\Gamma_k(x)|<(1-\delta_1)|\Gamma_{k-1}(x)|(n-|N_{k-1}(x)|)p-10) \leq \frac {2 } { n^{10}}.
\]
Which implies that, for $\epsilon_k$ as defined in \eqref{eqn:epsilonk} and  for large $n$,
\begin{align}\label{eqn:loewer}
	\Prob(|\Gamma_k(x)|<(1-\epsilon_k)|\Gamma_{k-1}(x)|np)\le \frac {2 } { n^{10}}.
\end{align}
Therefore \eqref{eqn:upperepsilon} and \eqref{eqn:loewer} give the result.
 \end{proof}

The following lemma and its proof are similar to Lemma~3 in \cite {bollobas1981diameter}. For the sake of completeness, we provide the proof here.
\begin{lemma} \label{lemma 6}
    Let  $\epsilon_k$ be as defined in \eqref{eqn:epsilonk}. For $1 \leq k \leq d-1$, define
    \begin{align*}
    \eta_k &= \exp \left( \sum\limits_{l=1}^{k} \epsilon_l \right) -1, \mbox{ and }
    \\\Omega_{k,x}^* &= \left \{ G \in \mathscr{G} [n]: \left\vert \vert \Gamma_{l}(x)\vert - n^l p^l  \right\vert \leq \eta_l  n^l p^l , \text { for all } 1 \leq l \leq k \right \},
    \end{align*}
   where $x$ is a fixed vertex. Then, for large $n$, we have
    \[
    \Prob(\Omega_{k,x}^* )\ge 1 - \frac{3 k}{n^{10}}.
    \]
\end{lemma} 
%\noindent Obseserve that, we have $\eta_k\to 0$ as $n\to \infty$ for $1\le k\le d-1$.
\begin{proof} [Proof of \cref{lemma 6}]
    Assume that $x$ is a fixed vertex. 
    %We define an event 
%    \begin{equation} \label{bounds of gamma l for all l}
%    \left\vert \vert \Gamma_{l}(x)\vert - n^l p^l  \right\vert \leq \eta_l  n^l p^l , \text { for all } 0 \leq l \leq k ,
%    \end{equation}
%    that is,
%    %\begin{equation}
%    	\Omega_{k,x}^* = \left \{ G \in \mathscr{G} [n]: \left\vert \vert \Gamma_{l}(x)\vert - n^l p^l  \right\vert \leq \eta_l  n^l p^l , \text { for all } 0 \leq l \leq k \right \}.
%    \end{equation}
   Recall $\Omega_{k,x}$ as in \eqref{def:Omega k,x}. Then $ \Omega_{k,x}^* \subseteq \Omega_{k-1, x}^* \subseteq \Omega_{k,x} $. 
%   We show the following using recursive relation
%    \begin{align}\label{eqn:inductionstep}
%    1- \Prob ( \Omega_{k,x}^* ) \leq  \frac{4 k }{n^{10} } .   
%    \end{align}
     Note that we have
\[\Omega_{k,x}^* =\Omega_{k-1 , x}^* \backslash \left \{  \left\vert \vert \Gamma_{k}(x)\vert - (np)^k  \right\vert \geq \eta_k (np)^k, \Omega_{k-1,x}^* \right \} .\]
Which implies that
\begin{equation} \label{eq 8}
     1- \Prob ( \Omega_{k,x}^* ) = 1 - \Prob ( \Omega_{k-1, x}^*) + \Prob (  \left\vert \vert \Gamma_{k}(x)\vert - (n p)^k  \right\vert \geq \eta_k (n p)^k, \Omega_{k-1, x}^* ).
\end{equation} 
Let $F_k=\{ \left\vert \vert \Gamma_{k}(x)\vert - (n p)^k  \right\vert \geq \eta_k (n p)^k\}$. 
%Since $\Omega_{ k,x}^*\subset \Omega_{ k,x}$, therefore we have 
%\[
%\Prob ( A_k\cap \Omega_{k-1, x}^* )\le \Prob (  A_k\cap \Omega_{k-1, x} )\le \Prob (  A_k\given  \Omega_{k-1, x} ).
%\]
%By the property of conditional probability we get
%\begin{align}
% \Prob (  A_k\cap \Omega_{k-1, x} )\le  \Prob (  A_k\given  \Omega_{k-1, x} )
%\end{align}
By the triangle inequality we have
\begin{align*}
	\left\vert \vert \Gamma_{k}(x)\vert - (n p)^k \right\vert \le \left\vert \vert \Gamma_{k}(x)\vert - \vert \Gamma_{k-1}(x) \vert n p \right\vert + 
	\left\vert \vert \Gamma_{k-1}(x) \vert n p - (n p)^k \right\vert.
\end{align*}
Observe that, for $G\in \Omega_{k-1,x}^* $, we have 
\begin{equation*} 
 \left\vert  \vert \Gamma_{k-1}(x) \vert n p- (n p)^k \right\vert 
\leq \left\vert (1+\eta_{k-1}) (n p)^{k} -(n p)^k   \right\vert 
\leq \eta_{k-1} (n p)^k .
\end{equation*}
Therefore using the last two equations, we conclude if $G\in A_k\cap \Omega_{k-1,x}^* $ then 
\[
\left\vert \vert \Gamma_{k}(x)\vert - \vert \Gamma_{k-1}(x) \vert n p  \right\vert \geq ( \eta_k - \eta_{k-1} ) (n p)^k\ge \epsilon_k (np)^k.
\]
Thus, as $\Omega_{k-1,x}^*\subseteq \Omega_{k-1,x}$, we obtain
\begin{align}\label{eqn:uselemma4}
\Prob ( F_k\cap \Omega_{k-1, x}^* )&\le  \Prob \left( \left\vert \vert \Gamma_{k}(x)\vert - \vert \Gamma_{k-1}(x) \vert n p  \right\vert \geq \epsilon_k (n p)^k ,\Omega_{k-1, x}^* \right)\nonumber
\\&\le  \Prob \left( \left\vert \vert \Gamma_{k}(x)\vert - \vert \Gamma_{k-1}(x) \vert n p  \right\vert \geq \epsilon_k (n p)^k ,\Omega_{k-1, x} \right)\nonumber
\\&\le  \Prob \left( \left\vert \vert \Gamma_{k}(x)\vert - \vert \Gamma_{k-1}(x) \vert n p  \right\vert \geq \epsilon_k (n p)^k \given \Omega_{k-1, x} \right)\nonumber
\\&\le \frac{3}{n^{10}}.
\end{align}
The last inequality follows by \Cref{lemma 4}. Therefore by \eqref{eq 8} and \eqref{eqn:uselemma4}  we get
\begin{align}\label{eqn:recursive}
\Prob \left( ( \Omega_{k,x}^* )^c \right)  \leq \Prob \left( ( \Omega_{k-1, x}^* )^c \right) + \frac{ 3}{n^{10}}.
\end{align}
By the Chernoff's bound it can be easily checked that, for large $n$,
\[
\Prob \left( ( \Omega_{1,x}^* )^c \right) \le \frac{3}{n^{10}}.
\]
Hence the result follows by the last equation and  the recursive relation in \eqref{eqn:recursive}.  
\end{proof}

\begin{proof}[Proof of \Cref{X alpha=1}]
  Suppose $\alpha = (x,y) \in I$ and $m \in \mathbb{N}$. Set 
  \begin{align} \label{definiton a,b}
  a'= (1-\eta_{d-1}) n^{d-1} p^{d-1} \text{ and } b'= (1+\eta_{d-1}) n^{d-1} p^{d-1}.
 \end {align} 
   Observe that, by the conditional probability, we have
  	\begin{align*} 
  	\begin{split}
  		\Prob \left( X_ \alpha =1 \right) 
  		 =& \sum\limits_ {m\in [a',b']} \Prob \left( X_ \alpha =1 \mid \left\vert \Gamma_{d-1}(x) \right\vert =m \right) \Prob \left( \left\vert \Gamma_{d-1}(x) \right\vert =m \right) \\
  		&+ \sum\limits_ {m\in [a',b']^c } \Prob \left( X_ \alpha =1 \mid \left\vert \Gamma_{d-1}(x) \right\vert =m \right) \Prob \left( \left\vert \Gamma_{d-1}(x) \right\vert =m \right) .
  	\end{split}
  	\end{align*}
Which further helps to obtain upper and lower bound as follows: 
\begin{align}
	\Prob \left( X_ \alpha =1 \right) 
	\leq& \sum\limits_ {m\in [a',b']} \Prob \left( X_ \alpha =1 \mid \left\vert \Gamma_{d-1}(x) \right\vert =m \right) \Prob \left( \left\vert \Gamma_{d-1}(x) \right\vert =m \right) \nonumber
	\\&+\sum\limits_ {m\in [a',b']^c } \Prob \left( \left\vert \Gamma_{d-1}(x) \right\vert =m \right) ,\label{upperbound of Xalpha}
\\	\Prob \left( X_ \alpha =1 \right) 
	 \geq& \sum\limits_ {m\in [a',b']} \Prob \left( X_ \alpha =1 \mid \left\vert \Gamma_{d-1}(x) \right\vert =m \right) \Prob \left( \left\vert \Gamma_{d-1}(x) \right\vert =m \right).\label{lowerbound of Xalpha}
\end{align}	
Note that $ \Prob \left( X_ \alpha =1 \mid \left\vert \Gamma_{d-1}(x)       \right\vert =m \right) = \Prob \left( d_G (x,y) > d \mid \left\vert \Gamma_{d-1}(x)       \right\vert =m \right)$ and $d_G (x,y) > d$ if and only if $y$ is not connected with any vertex of $\Gamma_{d-1} (x)$. Therefore,
\begin{align} \label{X alpha=1 given omega d-1}
 \Prob \left( X_ \alpha =1 \mid \left\vert \Gamma_{d-1}(x) \right\vert =m \right) 
		= (1-p) ^ { m}.
	\end{align}	
From the definition of $\Omega_{d-1,x}^*$ in \Cref{lemma 6}, it is clear that
\begin{align} \label{Omega subset of Gamma}
	\Omega_{d-1,x}^* \subseteq \{ a' \leq \left\vert \Gamma_{d-1}(x) \right\vert \leq b'\}.
\end{align} 	
Using \eqref{X alpha=1 given omega d-1}, \eqref{Omega subset of Gamma} in \eqref{upperbound of Xalpha} and \eqref{lowerbound of Xalpha}, we obtain
\begin{align*} 
		(1-p) ^ {b' }  \Prob \left( \Omega_{d-1,x}^* \right) 
		\leq \Prob \left( X_ \alpha =1 \right) 
		\leq  (1-p) ^ { a'}  + \Prob \left( (\Omega_{d-1,x}^*)^c \right).
\end{align*}
Next, using  $e^ {-p(1+p)} \leq 1-p \leq e^{-p}$ in the last inequality, we get
\begin{align} \label{eqprob bounds}
	 e^{-b'p(1+p)} \Prob \left( \Omega_{d-1,x}^* \right) \leq \Prob \left( X_ \alpha =1 \right) \leq e^{-a'p} + \Prob \left( (\Omega_{d-1,x}^*)^c \right) .
\end{align}	
Also,  from \cref{lemma 6}, we have
\begin{align}\label{eqn:1probomega}
	\Prob \left( \Omega_{d-1,x}^* \right) \geq 1 - \frac{3 (d-1) }{n^{10} } . 
\end{align}
Substituting the values of $a'$, $b'$ in \eqref{eqprob bounds}, and then using the assumption $n^{d-1}p^d=\log(n^2/c)$ and \eqref{eqn:1probomega}, we obtain
\begin{align*}
	 \left( \frac{c}{n^2} \right)^ {(1+p)(1+ \eta_{d-1} )} \left( 1- \frac{ 3 (d-1) }{n^{10} } \right)
	 \leq \Prob \left( X_ \alpha =1 \right)
	\leq \left( \frac{c}{n^2} \right)^{1- \eta_{d-1}} +  \frac{ 3(d-1) }{n^{10} }.
	\end{align*}
This gives the result, as $p\log n\to 0$ and $\eta_{d-1}\log n\to 0$ when $n\to \infty$.	 
\end{proof}

%%%%%%%%%%%%%%%%%%%%%%%%%%%%%%%%%%%%%%%%%%%%%%%%%%%%%%%%%%%%%
\subsection{Proof of \cref{X alpha and X beta=1} }
This subsection is dedicated to prove \cref{X alpha and X beta=1}. The following auxiliary lemmas will be used in the proof.
\begin{lemma} \label{lemma 8}
	Let $\eta_k$, and $\Omega_{ k,x}^*$ be as defined in \Cref{lemma 6}. For two vertices $x$ and $z$ we define $\Omega_{k,x,z}^*=\Omega_{k,x}^*\cap \Omega_{k,z}^*$. Then, for $1\le k\le d-1$,
	\[
	\Prob(\Omega_{k,x,z}^*)\ge 1-\frac{6 k}{n^{10}}.
	\]
\end{lemma}

\begin{proof} [Proof of \cref{lemma 8}]
	Using the union bound with \Cref{lemma 6}, we have
	\[ \Prob \left( (\Omega _{k,x,z}^*)^c \right) \leq \Prob \left( (\Omega _{k,x}^*)^c \right) +
	\Prob \left( (\Omega _{ k,z}^*)^c \right) \leq \frac{6 k}{n^{10} } .
	\]
Hence the result.
\end{proof}

\begin{lemma} \label{lemma:intersection}
	Let $ \Omega _{d-1,x,z}^*$ be as defined in \Cref{lemma 8}. Then, for large $n$,
	\[
	 \Prob \left(\left \vert \Gamma_{d-1} (x) \cap \Gamma_{d-1} (z) \right \vert \leq 10 n^{2d-3} p^{2d-2} \given \Omega _{d-1,x,z}^*\right)\ge 1 - {n^{-10}}.
	\]
\end{lemma}
\noindent The proof of \Cref{lemma:intersection} follows from \cite[Lemma 5]{bollobas1981diameter}. Hence we skip its proof.

\begin{proof}[Proof of \Cref{X alpha and X beta=1}]
	Let $\alpha ,\beta \in I$ be such that $\alpha = (x,y)$, $\beta = (z,w)$ with $\alpha\neq \beta$. Let $a'$ and $b'$ be defined as in \eqref{definiton a,b}.
To calculate the required probability, we take two cases.

\vspace{.1cm}
\noindent \underline {Case-I}: Suppose $\{x,y\}\cap \{w,z\}=\emptyset$. That is, all vertices $x,y,w,z$ are distinct. Let $m , m^\prime \in \mathbb{N}$.
Observe that, given $\left \vert \Gamma_{d-1} (x) \right\vert=m $ and  $\left \vert \Gamma_{d-1} (z) \right\vert =m^\prime$, the events $ \{ X_\alpha = 1 \}$ (resp. $\{ X_\beta = 1 \}$) occurs if there are no edges connected from $y$ (resp. $w$) to $\Gamma_{d-1}(x)$ (resp. $\Gamma_{d-1}(z)$). Therefore 
\begin{align} \label{X alpha=1 beta=1 given omega d-1}
	\Prob \left( X_ \alpha =1, X_\beta = 1 | \left \vert \Gamma_{d-1} (x) \right\vert=m, \left \vert \Gamma_{d-1} (z) \right\vert =m^\prime \right)
	&= \left( 1-p \right)^ { m + m^\prime}.
\end{align}
From the definition of $\Omega_{d-1,x,z}^*$ in \Cref{lemma 8} and $a',b'$ as in \eqref{definiton a,b}, it is clear that
\begin{align} \label{Omega subset of Gamma 2}
	\Omega_{d-1,x,z}^* \subseteq \{ a' \leq \left\vert \Gamma_{d-1}(x) \right\vert , \left\vert \Gamma_{d-1}(z) \right\vert \leq b' \}.
\end{align}	
Using similar type of inequality as in \eqref{upperbound of Xalpha}, \eqref{lowerbound of Xalpha}, and then by \eqref{X alpha=1 beta=1 given omega d-1}, we obtain the upper and lower bound of $\Prob \left( X_ \alpha =1,X_\beta = 1 \right) $ as follows:
	\begin{align}\label{upperbound of Xalpha Xbeta}
	\begin{split}
		\Prob \left( X_ \alpha =1,X_\beta = 1 \right) \leq \left( 1-p \right)^ { 2a'} + \sum\limits_ {m,m^\prime \in [a,b]^c}\Prob \left( \left \vert \Gamma_{d-1} (x) \right \vert =m, \left \vert \Gamma_{d-1} (z) \right \vert=m^\prime \right),
		%\sum\limits_ {m,m^\prime \in [a,b]} \left( 1-p \right)^ { m + m^\prime} \Prob \left( \left \vert \Gamma_{d-1} (x) \right \vert =m, \left \vert \Gamma_{d-1} (z) \right \vert=m^\prime \right) \\
		%+ \sum\limits_ {m,m^\prime \in [a,b]^c}\Prob \left( \left \vert \Gamma_{d-1} (x) \right \vert =m, \left \vert \Gamma_{d-1} (z) \right \vert=m^\prime \right)
	\end{split}
\end{align} 
	\begin{align}\label{lowerbound of Xalpha Xbeta}
	\Prob \left( X_ \alpha =1,X_\beta = 1 \right) 
	\geq  \left( 1-p \right)^{2b'} \sum\limits_ {m,m^\prime \in [a',b']} \Prob \left( \left \vert \Gamma_{d-1} (x) \right \vert =m, \left \vert \Gamma_{d-1} (z) \right \vert=m^\prime \right).
	%\sum\limits_ {m,m^\prime \in [a,b]} \left( 1-p \right)^ { m + m^\prime} \Prob \left( \left \vert \Gamma_{d-1} (x) \right \vert =m, \left \vert \Gamma_{d-1} (z) \right \vert=m^\prime \right).
\end{align}
Using the inequality $e^ {-p(1+p)} \leq 1-p \leq e^{-p}$ and \eqref{Omega subset of Gamma 2} in \eqref{upperbound of Xalpha Xbeta} and \eqref{lowerbound of Xalpha Xbeta}, we obtain
\begin{align*}
  e^ {-2b'p(1+p)} \Prob \left( \Omega_{d-1,x,z}^* \right) \leq \Prob \left( X_ \alpha =1,X_\beta = 1 \right) \leq e^{-2a'p} +  \Prob \left( (\Omega_{d-1,x,z}^* )^c \right),
\end{align*}
where $a',b'$ are as in \eqref{definiton a,b}. Consequently, by a similar argument as in the proof of \Cref{X alpha=1}, and by \cref{lemma 8} we get
\begin{align*}
	\left( \frac{c^2}{n^4} \right)^{(1+p)(1+ \eta_{d-1} )} 
	\left( 1- \frac{6 (d-1) }{n^{10} } \right) 
	\leq \Prob \left( X_ \alpha =1,X_\beta = 1 \right) \leq \left( \frac{c^2}{n^4} \right)^ {1- \eta_{d-1}} +\frac{6 (d-1) }{n^{10} }. 
\end{align*}
Which gives the result, in this case, as $p\log n\to 0$ and $\eta_{d-1}\log n\to 0$ when $n\to \infty$.

\vspace{.1cm}
\noindent \underline {Case-II}: Suppose $\{x,y\}\cap \{w,z\}\neq \emptyset$. That is, all vertices $x,y,w,z$ are not distinct. Without loss of generality we assume that $y = w$. Then we have 
\[
\Prob \left( X_ \alpha =1, X_\beta = 1 | \left \vert \Gamma_{d-1} (x)\cup  \Gamma_{d-1} (z) \right \vert =m \right)
= \left( 1-p \right)^ {m }.
\]
	Given $\Omega _{d-1,x,z}^*$,  using \Cref{lemma:intersection} we have
\begin{align} \label{bounds of union}
2 \left( 1 - \eta_{d-1} \right) (p n)^{d-1} \left( 1-o(1) \right)
\leq \left \vert \Gamma_{d-1} (x) \cup \Gamma_{d-1} (z) \right \vert 
\leq 2 \left( 1 + \eta_{d-1} \right) (p n)^{d-1},
\end{align}
with probability at least $  1 - n^{-10}$. Which further implies that \eqref{bounds of union} holds with probability at least $ \left( 1 - n^{-10} \right) \Prob \left( \Omega _{d-1,x,z}^*\right) $. By \Cref{lemma 8}, we have
\[
 \left( 1 - n^{-10} \right) \Prob \left( \Omega _{d-1,x,z}^*\right) \geq 1 - \frac{1}{n^{9}} 
\]
Thus by a similar argument as in \underline{Case-I}, we obtain
 \begin{align*}
 	\left( \frac{c^2}{n^4} \right)^{(1+p)(1+ \eta_{d-1} )} 
 	\left( 1- \frac{1 }{n^{9} } \right) 
 	\leq \Prob \left( X_ \alpha =1,X_\beta = 1 \right) \leq \left( \frac{c^2}{n^4} \right)^ {(1- \eta_{d-1})  \left( 1-o(1) \right)} +\frac{1 }{n^{9} }. 
 \end{align*}
 Which gives the result, in this case, as $p\log n\to 0$ and $\eta_{d-1}\log n\to 0$ when $n\to \infty$.
\end{proof}

%%%%%%%%%%%%%%%%%%%%%%%%%%%%%%%%%

\section{Proof of \Cref{main theorem}}\label{sec:HG}
In this section, we first present the proof of \cref{main theorem} using \Cref{proposition 2} and the following lemma, which is analogous to \Cref{probability comparision}. The proof of the lemma is given at the end of this section and the proof of \Cref{proposition 2} is provided in the next section. Throughout this section and the next, we focus on the case $t\geq 3$, since the case $t=2$ was treated separately in the earlier sections.

\begin{lemma} \label{lemma 25}
	Let $ 0 < p_1 \le p_2 \leq 1$, and $r$ be a positive integer and $\mathcal H\in \mathscr{H}[n]$, then 
	\[
	\Prob _1 \left( diam(\mathcal{H})\leq r \right) \leq \Prob _2 \left( diam(\mathcal{H})\leq r \right) ,
	\]
	where $\Prob _i $ denotes the probability in the space $ \mathscr{H} (n,t,p_i ) ,$ for $1 \leq i \leq 2$.
\end{lemma}

\begin{proof} [Proof of \cref{main theorem}]
The result follows from  \Cref{lemma 25} and \Cref{proposition 2} using the same arguments as in the proof of \Cref{Bollobas theorem}. We skip the details.
\end{proof}

Now we prove \cref{lemma 25}. The result follows using the same arguments as in the proof of \Cref{probability comparision}. For the sake of completeness, we outline the proof.
\begin{proof} [Proof of \cref{lemma 25}]
	Let $ \mathcal{E} = \{ e_1, e_2,\ldots ,e_{ {n \choose t}} \}$ denote the set of all $ {n \choose t} $ hyperedges of size $t$ on the vertex set $[n]=\{ 1,2,...,n\}$.  Consider that $ X= \left( X_{e} : e \in \mathcal{E} \right)$ and $ Y= \left( Y_{e} : e \in \mathcal{E} \right)$ are two random variables, where $ \{X_{e}\, ; e \in \mathcal{E} \}$ are i.i.d. Bernoulli($p_1$) and $ \{Y_{e}\,; e \in \mathcal{E}\}$  are i.i.d.  Bernoulli($p_2$) random variables. Consider $ Z= \left( Z_{e}: e \in \mathcal{E} \right)$
	on $ [0,1]^{n \choose t} $, where $ \{Z_{e}\, ; e \in \mathcal{E} \}$ are i.i.d. uniform random variables on $ [0,1]$. We define following random variables
	\[
	X^{\prime}= \left( \mathbbm{1}_ { \{ Z_{e} \leq p_1 \} } : e \in \mathcal{E} \right)  \mbox{ and } Y^{\prime}= \left( \mathbbm{1}_ { \{ Z_{e} \leq p_2 \} } : e \in \mathcal{E} \right).
	\]
%	as follows: for $ w = \left( w_{e}: e \in \mathcal{E} \right) \in [0,1]^{n \choose t} $,
%	\begin{align*}
%		X^{\prime} (w) = & \text{a hypergraph } \mathcal{H} \text{ with vertex set } [n] \text{ and includes edges } e \text{ if } \mathbbm{1}_ { \{ Z_{e} \leq p_1 \} } (w_{e}) =1, \\
%		& \text{for } e \in \mathcal{E} .
%	\end{align*}
%	Then,
%	\begin{align*}
%		\Prob _1 (\mathcal{H}) &= \Prob \circ \left( X^{\prime} \right)^{-1} (\mathcal{H}) \\
%		&= \Prob \{ w \in [0,1]^{n \choose t} : X^{\prime} (w) = \mathcal{H} \} \\
%		&= \Prob \left \{ \left( w_{e}: e \in \mathcal{E} \right) \in [0,1]^{n \choose t} : \text{for } e \in \mathcal{E}, \mathbbm{1}_ { \{ Z_{e} \leq p_1 \} } (w_{e}) =
%		\begin{cases}
%			1, \mathcal{H} \text{ has hyperedge } e \\
%			0, \text{ otherwise}
%		\end{cases}
%		\right \} \\
%		&= p_1^{ \vert \mathcal {E} ( \mathcal{H} ) \vert } 
%		\left( 1- p_1 \right) ^{ {n \choose t} - \vert \mathcal {E} ( \mathcal{H} ) \vert } .
%	\end{align*}
%	Consequently, the range space is $ \mathscr{H}(n,t,p_1)$. Similarly, for $p_2$, we define the random variable $Y^{\prime}= \left( \mathbbm{1}_ { \{ Z_{e} \leq p_2 \} } : e \in \mathcal{E} \right)$ and we obtain the range space $ \mathscr{H}(n,t,p_2)$. 
It is easy to see that $ X^{\prime} \overset{d}{=} X $ and $ Y^{\prime} \overset{d}{=} Y $. Therefore $ \left( X^{\prime}, Y^{\prime} \right)$ is a coupling of $X$ and $Y$. Clearly,
	\(
	\mathbbm{1}_ { \{ Z_e \leq p_1 \} } \left( w_e \right) \leq  \mathbbm{1}_ { \{ Z_e \leq p_2 \} } \left( w_e \right), \text{ for all }  w_e \in [0,1] .
\)
	Thus  $X^{\prime} (w)$ is a subgraph of $Y^{\prime} (w)$ for each $w =\left( w_e \right) \in [0,1]^{ \binom {n}{t} } $.
%	Thus, for $p_1 \leq p_2$ and a fixed $(x,y) \in I$,
%	\begin{align*}
%		\Prob_1 \left( \{ \mathcal{H}\in \mathscr{H}[n] :d_\mathcal{H} (x,y) \leq r \} \right) &= \Prob \circ (X^{\prime} )^{-1}\left( \{ \mathcal{H}\in \mathscr{H}[n] :d_\mathcal{H} (x,y) \leq r \} \right) \\
%		&= \Prob \left( \{  w \in [0,1]^{ \binom {n}{2} } :X^{\prime}(w )= G \text{ s.t. } d_\mathcal{H} (x,y) \leq r \} \right) \\
%		&\leq \Prob \left( \{  w \in [0,1]^{ \binom {n}{2} } :Y^{\prime}(w )= G \text{ s.t. } d_G (x,y) \leq r \} \right) \\
%		&= \Prob \circ (Y^{\prime} )^{-1}\left( \{ \mathcal{H}\in \mathscr{H}[n] :d_\mathcal{H} (x,y) \leq r \} \right) \\
%		&= \Prob_2 \left( \{ \mathcal{H}\in \mathscr{H}[n]:d_\mathcal{H} (x,y) \leq r \} \right).
%	\end{align*}
	Consequently, we have 
	\[
	\Prob_1 \left( d_\mathcal{H} (x,y) \leq r \text{ for all } (x,y) \in I \right) \leq \Prob_2 \left( d_\mathcal{H} (x,y) \leq r \text{ for all } (x,y) \in I \right) .
	\]
	Hence the result.
\end{proof}

\section{Proof of \cref{proposition 2}}  \label{sec:prop2}  
This section is dedicated to proving \Cref{proposition 2}. We follow the similar steps as in the proof of \Cref{proposition 1}. However, due to the complexity  in the structure of the model, the computations  will be challenging in this case. The following lemmas will be used in the proof of the proposition.

\begin{lemma} \label{X alpha=1 in hypergraph}
	Let $t,c,d,n,p$ be as in Assumption~\ref{ass main} and $\mathcal{H} \in \mathscr{H} (n,t,p)$.  Suppose  $\alpha = (x,y)$, where $x$ and $y$ are two vertices and  
	\begin{equation} \label{def X alpha in hypergraph}
		X_ \alpha =
		\begin{cases}
			1 & \text{when } d_\mathcal{H} (x,y)>d \\
			0 & \text{otherwise.}
		\end{cases}
	\end{equation}
	Then, for each $\alpha\in I$, we have
	\[
	\Prob \left( X_ \alpha =1 \right) \approx \frac{c}{n^2} , \mbox{ as $n\to \infty$}.
	\]
\end{lemma}

\begin{lemma} \label{X alpha=1 beta=1 in hypergraph}
	Let $t,c,d,n,p$ be as in Assumption~\ref{ass main} and $\mathcal{H} \in \mathscr{H} (n,t,p)$. Suppose $I$ is the index set as defined in \eqref{def I}. Then, for  $ \alpha , \beta \in I$ and $\alpha \neq \beta$ 
	\[
	\Prob \left( X_ \alpha =1,X_\beta = 1 \right) \approx \frac{c^2}{n^4}, \mbox{ as $n\to \infty$},
	\]
	where $X_\alpha$ is as defined in \eqref{def X alpha in hypergraph}.
\end{lemma}

\begin{lemma} \label{stochastic dominance lemma for hypergraph}
	Let   $ \{ X_{\alpha} \}_{\alpha \in I}$ are the  Bernoulli random variables defined as in \eqref{def X alpha in hypergraph}. Consider the random variables $W_n,J$ and $W_n^*$  as in Fact~\ref{W size biased lemma}, then
	\[
	W_n +1 - X_J \leq _{st} W_n^*.
	\] 
\end{lemma}

The proofs of \cref{stochastic dominance lemma for hypergraph} and \Cref{proposition 2} are similar to those of \Cref{lem:positivelyrelated} and \Cref{proposition 1}. Therefore, we will first provide an outline of these proofs. The proofs for \Cref{X alpha=1 in hypergraph} and \Cref{X alpha=1 beta=1 in hypergraph} will be provided in the following two subsections.

\begin{proof}[Proof of \cref{stochastic dominance lemma for hypergraph}]
	Recall, the probability space 	$\mathscr{H} \left( n,t,p \right) = \left( \mathscr{H} [n], \mathscr{F}, \Prob \right)$, where the random variables $ \{ X_{\alpha} \}_{\alpha \in I}$ are defined. Observe that 	$\mathscr{H} \left( n,t,p \right)$ can be viewed as the space $\Omega = \{ 0,1\}^ { \binom {n}{t} }$ with the Bernoulli product measure. In the set $\Omega$, we define the partial order relation as follows:
	\[
	w \preceq w^\prime \text { if } w(u)\geq w^\prime (u) , \text{ for all } u \in \left[\binom{n}{t} \right].
	\]
	Similar to the graph case, it is easy to see that $\{ X_\alpha =1 \}$ is an increasing event with respect to this partial order. Hence, applying the Harris-FKG inequality,
	\begin{equation*}
		\Prob \left( X_\beta =1 \right) \Prob \left( X_\alpha = 1 \right) \leq \Prob \left( X_\beta =1 , X_\alpha = 1 \right).
	\end{equation*}
	 Which implies that, for fixed $\alpha \in I$,
	\begin{equation*} \label{X beta stochastically smaller}
		X_\beta \leq _{st} X_\beta^\alpha, \text { for all } \beta \neq \alpha.
	\end{equation*}
Note that the random variables $ \{ X_{\alpha} \}_{\alpha \in I}$ are exchangeable. The rest of the proof follows as in the proof of \Cref{lem:positivelyrelated}. We skip the details.
	\end{proof}

\begin{proof}[Proof of \Cref{proposition 1}]
	Recall $W_n $ denotes the total number of remote pairs in hypergraphs. To prove the result, we show that
	\[
	d_{TV} (W_n,Poi(c/2)) \to 0 \,\, \mbox{ as } n\to \infty.
	\] 
	By the  triangle inequality, we have 
	\begin{equation*} 
			d_{TV} \left(W_n,Poi( c/2 )\right) \leq d_{TV} \left(W_n,Poi(\E W_n)\right) + d_{TV} \left(Poi(\E W_n), Poi(c/2)\right).
		\end{equation*}
		The rest of the proof follows by combining Fact~\ref{ft:TVbound}, Fact~\ref{fact1}, \Cref{X alpha=1 in hypergraph} and \Cref{X alpha=1 beta=1 in hypergraph} and \Cref{stochastic dominance lemma for hypergraph}, similar to the proof of \Cref{proposition 1}. We omit the details.
%Similar to the graph case, \cref{X alpha=1 in hypergraph} implies 
%\[
%\E [W_n] %=|I|\Prob \left( X_ \alpha =1 \right) 
%\approx \frac{c}{2},\;\; \mbox{ as } n\to \infty.
%\]
%The last inequality and Fact~\ref{fact1} together implies
%\begin{align}\label{eqn:2ndterm}
%	d_{TV} \left(Poi(\E [W_n]), Poi(c/2)\right) \to 0 \text{ as } n \to \infty.
%	% &\leq \left \vert \E [W_n] - c/2 \right \vert 		
%	%& \leq  \frac{n(n-1)}{2}\left[ (\frac{c}{n^2})^ { (1 - \eta_{d-1})} ( 1 + o(1)) \right] -\frac{c}{2} 
%\end{align} 
%To estimate the first term in the right hand side of \eqref{eq 3}, we use \eqref{eq 4} for large $n$. Recall $W_n=\sum_IX_\alpha$. Then similar to graph case, from \Cref{X alpha=1 in hypergraph} and \Cref{X alpha=1 beta=1 in hypergraph} we get 
%\[
%\E (W_n^2)\approx \frac{c}{2} +\frac{c^2}{4}, \mbox{ as } n\to \infty.
%\]
%Putting the values of $\E W_n$, $ \E (W_n ^2)$ and $ \E (X_\alpha) $ in the \eqref{eq 4} we get,
%\begin{align}\label{eqn:1stterm}
%	d_{TV} \left( W_n,Poi(\E W_n) \right) 
%	%&\leq \left( \vert I \vert^2 - \vert I \vert \right) \left[ \left( \frac{c^2}{n^4} \right) ^{ \left( 1 - \eta_{d-1} \right) \left( 1-o(1) \right) } \left( 1+o(1) \right) \right] - \\
%	%	& \vert I \vert ^2 \left[ (\frac{c}{n^2})^ { (1+ \eta_{d-1}) (p+1)} ( 1 - o(1)) \right]^2 + 2 \vert I \vert \left[ (\frac{c}{n^2})^ { (1 - \eta_{d-1})} \left( 1 + o(1) \right) \right]^2 \\
%	& \to 0 \text{ as } n \to \infty .
%\end{align}
%Therefore using \eqref{eqn:2ndterm} and \eqref{eqn:1stterm} in \eqref{eq 3} we get the result.
\end{proof}

\subsection{Proof of \cref{X alpha=1 in hypergraph}}
Note that the proofs of the following lemmas are given under Assumption \ref{ass main}. The proof of \cref{X alpha=1 in hypergraph} is preceded by several supporting lemmas. Note that for a fixed vertex $x$, the {\it star} of $x $ is defined by
\(
H(x)= \{e \in \mathcal{E}: x\in e \}, 
\)
and we write $H_1 (x) = H(x)$.
%\[
%\vert V\left(H(x)\right) \vert = \Big | \bigcup_{e \in H(x)} V(e) \Big | = \Big| { \bigcup_{e \in H(x)}e} \Big | 
%\]
%denotes the number of vertices in $H(x)$.
Moreover for $k \geq 2$, we define
%\[
%H_k(x)= \{e \in \mathcal{E}: e \cap e^\prime \neq \emptyset, \text{ for some } e^\prime \in H_{k-1}(x) \text{ and } e \cap e^{\prime \prime} = \emptyset, \forall e^{\prime \prime} \in H_{k-2}(x) \} ,
%\]
%equivalently,
\[
H_k(x)= \{e \in \mathcal{E}: e \cap \Gamma_{k-1} (x) \neq \emptyset, e \cap N_{k-1}^c (x) \neq \emptyset \text{ and } e \cap N_{k-2} (x)= \emptyset \} .
\]
For a positive integer $k$, $ 1\leq k \leq d -1$, and a fixed vertex $x$, we define $\Omega_{k,x} $ to be the set of hypergraphs having $ \vert \Gamma_{k-1}(x) \vert $ and $ \vert N_{k-1}(x) \vert $  as follows:
\begin{equation} \label{hyp gamma_k-1}
\frac{1}{2}((t-1) Np)^{k-1} \leq \vert \Gamma_{k-1} (x) \vert \leq \frac{3}{2}((t-1) Np)^{k-1}, 
\end{equation}
recall $N = {n-1\choose t-1} $, and 
\begin{equation} \label{hyp N_k-1}
\vert N_{k-1} (x) \vert \leq 2((t-1) Np)^{k-1}.
\end{equation}
In other words, the set $\Omega_{k,x} $  (with an abuse of notaion) can be written as
\begin{equation} \label{def Omega k,x for hyp}
\Omega_{k,x} := \left\{ \mathcal{H} \in \mathscr{H}[n]: \mathcal{H} \text{ satisfies } \eqref{hyp gamma_k-1} \text{ and } \eqref{hyp N_k-1} \right\} .
\end{equation}

\begin{lemma} \label{lemma 14}
Let $t,c,d,n,p$ and  $N$ be  as in Assumption \ref{ass main}. Let $L \geq 72 (t-1) $ be a constant and $x$ be a fixed vertex. Given $\vert \Gamma_{k-1}(x) \vert$ and $\vert N_{k-1}(x) \vert$, define 
\[
a_k=\sum_{m=1}^{t-1} {\vert \Gamma_{k-1}(x) \vert \choose m} { n- \vert N_{k-1}(x) \vert\choose t- m} p \mbox{ and } \delta_k = \left [ \frac { L \log n}{  (t-1)^k N^k p^k  } \right ]^{1/2},
\]
for $ 1\leq k \leq d -1$. Then, for sufficiently large n, we obtain 
\[
\Prob \left(\l||H_k(x)|-a_k\r|>\delta_ka_k \given \Omega_{ k,x}\right)\le 2 n^ {- 6}.
\]

%Then, given $\Omega_{k,x}$, for large n, we obtain 
%\begin{equation*}
%	\begin{split}
%		(1- \delta_k) \sum_{m=1}^{t-1} {\vert \Gamma_{k-1}(x) \vert \choose m} { n- \vert N_{k-1}(x) \vert\choose t- m} p 
%		\leq \vert H_k(x) \vert \leq  \\
%		(1+ \delta_k) \sum_{m=1}^{t-1}{\vert \Gamma_{k-1}(x) \vert \choose m} { n- \vert N_{k-1}(x) \vert\choose t- m} p ,
%	\end{split}
%\end{equation*}
%\[
%\text{ where } \delta_k = \left [ \frac { L \log n}{  (t-1)^k N^k p^k  } \right ]^{1/2} ,
%\]
%with probability at least $ 1 - 2 n^ {- \frac{L}{12(t-1)}} $.
\end{lemma}
\begin{proof} [Proof of \cref{lemma 14}]
Observe that, given $|\Gamma_{k-1}(x)|$ and $|N_{k-1}(x)|$, the random variable $\vert H_k(x) \vert$ has Binomial distribution with parameters $a_k$ and $p$. Hence, 
\[
\E \vert H_k(x) \vert = \sum_{m=1}^{t-1} {\vert \Gamma_{k-1}(x) \vert \choose m} { n- \vert N_{k-1}(x) \vert\choose t- m} p,
\]
as $m$ points of each edge are chosen from $\Gamma_{k-1}(x)$ and the rest of the $(t-m)$ points are chosen from $N_{k-1}(x)^c$.
Applying Chernoff bound,  given $|\Gamma_{k-1}(x)|$ and $|N_{k-1}(x)|$, we obtain
\begin{align}\label{eqn:chernoff14}
	\Prob(|H_k(x)-a_k|>\delta_ka_k)\le 2e^{-\frac{\delta_k^2a_k}{3}}
\end{align}
%\[
%\left\vert \vert H_k(x) \vert - \sum_{m=1}^{t-1} {\vert \Gamma_{k-1}(x) \vert \choose m} { n- \vert N_{k-1}(x) \vert\choose t- m} p \right\vert
%\leq 
%\delta_k \sum_{m=1}^{t-1}{\vert \Gamma_{k-1}(x) \vert \choose m} { n- \vert N_{k-1}(x) \vert\choose t- m} p ,
%\]
%with probability at least 
%$  1 - 2 \exp \left[ - \frac{\delta_k ^2}{3} 
%\sum\limits_{m=1}^{t-1} {\vert \Gamma_{k-1}(x) \vert \choose m} { n- \vert N_{k-1}(x) \vert\choose t- m} p \right]$. 
Observe that, given $\Omega_{ k,x}$,  we have 
\[
|\Gamma_{k,x}|=o(n), |N_{k,x}|=o(n) \mbox{ and } { n- \vert N_{k-1}(x) \vert\choose t- 1} = { n- 1 \choose t- 1} (1 - o(1)), \mbox{ as $n\to \infty$,}
\]
for $1\le k\le d-1$. Therefore, given $\Omega_{k,x}$, we have
\begin{align*}
\sum\limits_{m=1}^{t-1} {\vert \Gamma_{k-1}(x) \vert \choose m} { n- \vert N_{k-1}(x) \vert\choose t- m} p &= \vert \Gamma_{k-1}(x) \vert  { n- 1\choose t- 1} p  \left(1 + o(1) \right)
\\&= \vert \Gamma_{k-1}(x) \vert N p  \left(1 + o(1) \right), \mbox{ as $n\to \infty$.}
\end{align*}
Which implies that given the event $\Omega_{k,x}$, for large $n$,
\begin{align*}
	\delta_k^2a_k\ge \frac{L\log n}{4(t-1)}.
\end{align*}
Hence the result follows from \eqref{eqn:chernoff14}.
\end{proof}
Next two lemmas will be useful in calculating the lower bound of $ \vert \Gamma_k (x) \vert$. Let $H_k ^1(x)$ denote the set of hyperedges in $H_k(x) $ that consist exactly one vertex from $\Gamma_{k-1}(x) $ and remaining $t-1$ vertices from $ V(\mathcal{H}) \backslash N_{k-1}(x)$, that is,
\[
H_k ^1(x):= \{ e \in H_k(x) : \vert e \cap \Gamma_{k-1}(x) \vert = 1 \}.
\]

\begin{lemma} \label{lemma 15}
Let $t,c,d,n,p$ and  $N$ be as  in Assumption \ref{ass main} and $\delta_k$  be as in \Cref{lemma 14}. Let $x$ be a fixed vertex. 
Suppose, given $\Gamma_{k-1}(x)$,
\[
b_k= \vert \Gamma_{k-1}(x) \vert  { n- \vert N_{k-1}(x) \vert\choose t- 1}p, \mbox{ for $1\le k\le d-1$}.
\]
Then, for large n, we obtain 
\[
\Prob \left(\l||H_k ^1(x)|-\E[|H_k ^1(x)|]\r|>\delta_k\E[|H_k ^1(x)|]\given \Omega_{k,x}\right)\le n^{-6}.
\]
%\[
%(1- \delta_k )b_k\leq \vert H_k ^1(x) \vert \leq (1+ \delta_k ) b_k , 
%\]
%\[
%\text{ where } \delta_k = \left [ \frac { L \log n}{  (t-1)^k N^k p^k  } \right ]^{1/2} ,
%\]
%with probability at least $  1 - 2 n^ {- \frac{L}{12(t-1)}} $.
\end{lemma}
\begin{proof} [Proof of \cref{lemma 15}]
Note that, given $|\Gamma_{k-1}(x)|$ and $|N_{k-1}(x)|$, the random variable $ \vert H_k ^1(x) \vert $ has Binomial distribution with parameters $\vert \Gamma_{k-1}(x) \vert  { n- \vert N_{k-1}(x) \vert\choose t- 1}  $ and $p$. Therefore,  given $|\Gamma_{k-1}(x)|$ and $|N_{k-1}(x)|$, the Chernoff bound with $ \delta_k$ yields
\[
\Prob \left(\left \vert \vert H_k ^1(x) \vert - b_k  \right \vert>\delta_k b_k\right)
\leq e^{-\frac{\delta_k^2b_k}{3}}.
\]
%Given $\Omega_{k,x}$, 
%\begin{align*}
%	\exp\left( -\frac{\delta_k^2 }{3} \vert \Gamma_{k-1}(x) \vert  { n- \vert N_{k-1}(x) \vert\choose t- 1}p \right) 
%	&\leq \exp \left( -\frac{\delta_k^2 }{6} \{ (t-1)Np \}^{k-1} N (1 -o(1))p \right)\\
%	& \leq  e^{- \frac{L \log n}{12(t-1)} } = n^ {- \frac{L}{12(t-1)}},
%\end{align*}
 The result follows by estimating  $\delta_k$ and $b_k$ as in the proof of \Cref{lemma 14}.
\end{proof}    
For a fixed vertex $ v \in  \Gamma_{k-1} (x)$, the set of hyperedges containing $v$ and remaining $(t-1)$ vertices from $ N_{k-1}^c (x)$ is denoted by $ H_k^1 (x,v)  $, and the set of vertices of the hyperedges of $ H_k^1 (x,v) $  is denoted by $ V( H_k^1 (x,v))$, that is, 
\[
H_k^1 (x,v) :=  \{ e \in H_k^1 (x) : v\in e \} \mbox{ and } V ( H_k^1 (x,v) ) :=  \bigcup_{e \in H_k^1 (x, v)}  V(e).
\] 
Note that $ \vert H_1^1 (x,x) \vert = \deg (x) $ and $ \bigcup \limits _{v \in  \Gamma_{k-1} (x) }  H_k^1 (x,v) = H_k^1 (x)$. Also, we define
\[
\vert H_k^{1,0} (x,v) \vert = \sup \{ i: \exists i \text { pairwise disjoint sets among } e_1 \backslash \{ v\}, \ldots ,e_l\backslash \{ v\}\}.
\]
where $H_k^1 (x,v)=\{e_1,\ldots,e_l\}$ (say).
\begin{lemma} \label{lemma 16}
Let $t,c,d,n,p$ and  $N$ be as in Assumption \ref{ass main} and $\delta_1$  be as in \Cref{lemma 14}. Let $1 \leq k \leq d -1$ and $ v \in  \Gamma_{k-1} (x) $ be a fixed vertex. 
Suppose 
\[
c_k={ n- \vert N_{k-1}(x) \vert\choose t- 1}p.
\]
Then, for large $n$, we have 
\begin{align}
&\Prob(|H_k^1(x,v)-c_k|>\delta_1c_k\given \Omega_{k,x})\le 2 n^ { - 12 } \mbox{ and }\label{eqn:lem161}
\\& \Prob(\vert H_k^{1,0} (x,v) \vert\ge  (1- \delta_1 )c_k-10\given \Omega_{k,x})\ge 1-  \frac{ 4(t-1)^{11} }{n^{11} } .\label{eqn:lem162}
\end{align}
%As a consequence we have 
%\begin{align*}
%\Prob(\{|H_k^{1,0}(x,v)|\ge (1-\delta_1)c_k, v\in \Gamma_{k-1}(x)\}\given \Omega_{k,x})\ge 1-|\Gamma_{k-1}(x)|\frac{ 4(t-1)^{10} }{n^{11}}.
%\\\Prob((t-1) \left[ (1- \delta_k^1 )c_k -10 \right] \leq \vert V( H_k^1 (x,v)) \given \Omega_{k,x})\ge 1-|\Gamma_{k-1}(x)|\frac{ 4(t-1)^{10} }{n^{11}}.
%\end{align*}
%Set
%\[
%\delta_k^1 = \left [
%\frac {L \log n}{(t-1) N p } \right ]^{1/2}, \text{ where } L \geq 72 (t-1) .
%\]
%Then, for given $\Omega_{k,x}$, 
%\[
%(1- \delta_k^1 ){ n- \vert N_{k-1}(x) \vert\choose t- 1}p \leq \vert  H_k^1 (x,v) \vert \leq (1+ \delta_k^1 ) { n- \vert N_{k-1}(x) \vert\choose t- 1}p , 
%\]
%and 
%\[
%\vert H_k^{1,0} (x,v) \vert \geq (1- \delta_k^1 ){ n- \vert N_{k-1}(x) \vert\choose t- 1}p -10
%\]
%occur with probability at least
%\[
%1- 2 n^ { - \frac{L}{6(t-1)} } - \frac{ 4(t-1)^{11} }{n^{11} } .
%\]
\end{lemma}
\begin{proof} [Proof of \cref{lemma 16}]
Observe that, given $|N_{k-1}(x)|$,  the random variable $ \vert  H_k^1 (x,v) \vert$ has Binomial distribution with parameters $ { n- \vert N_{k-1}(x) \vert\choose t- 1} $ and $p$. Then \eqref{eqn:lem161} holds by the Chernoff bound with $\delta_1$ and the similar calculation as in \Cref{lemma 14}. We skip the details and proceed to prove \eqref{eqn:lem162}. 
%Let
%\[
%A = \{ \mathcal{H} \in \mathscr{H} [n]: \mathcal{H} \text{ satisfies } \eqref{eq 14} \},
%\]
%then, given $\Omega_{k,x}$, we obtain 
%\begin{align*}
%	\Prob (A) \geq 1- 2 \exp \left( -\frac{(\delta_k^1 )^2 }{3} { n- \vert N_{k-1}(x) \vert\choose t- 1}p \right) 
%	&= 1- 2 \exp \left( -\frac{L \log n }{3 (t-1)} \left( 1-o(1) \right) \right)\\
%	& \geq 1- 2 n^ { - \frac{L}{ 6(t-1)} } .
%\end{align*}

For $m \in \mathbb{N}$ and $x,v$ as mentioned above,  we define
\begin{align*}
A_m &= \{ \mathcal{H} \in \mathscr{H} [n]: \vert H_k^1 (x,v) \vert = m \},
\\A_m^0& = \{ \mathcal{H} \in A_m : \vert H_k^{1,0} (x,v) \vert \geq m-10 \},
\\B_m &= \{ \mathcal{H} \in A_m: \vert H_k^{1,0} (x,v) \vert = m \}.
\end{align*}
Let $A=\{ \mathcal{H} \in \mathscr{H} [n]\suchthat (1-\delta_1)c_k\le |H_k^1(x,v)|\le (1+\delta_1)c_k\}$, where $c_k$ is as defined in \Cref{lemma 16}. Then, as the events $A_m $ are disjoint, we have 
\[ 
\Prob (A) = \sum\limits_{ m= (1- \delta_1 )c_k}^{ (1+ \delta_1 ) c_k} \Prob (A_m).
\]
Next, using the fact that $\Prob (A_m^0 ) = \Prob ( A_m) - \Prob (A_m \cap (A_m^0)^c)$, we obtain
\begin{equation} \label{Probability of A m naught}
	\Prob (A_m^0 ) = \Prob ( A_m) \left[1 - \frac {\vert A_m \cap (A_m^0)^c \vert} {\vert A_m \vert} \right] \geq \Prob ( A_m) \left[1 - \frac {\vert A_m \cap (A_m^0)^c \vert} {\vert B_m \vert} \right].
\end{equation}
The last inequality is a consequence of the fact that $B_m \subset A_m$. If $H_k^1 (x,v)=\{e_1,\ldots,e_l\}$ and  $e_1\backslash \{ v\}, \ldots, e_l\backslash \{v\}$ are disjoint sets, then we have
\begin{align*}
	\vert B_m \vert 
%&=\left \vert  \{ \mathcal{H} \in A_m  : \vert H_k^{1,0} (x,v) \vert = m \} \right \vert  \\
	&= { n- \vert N_{k-1}(x) \vert\choose t- 1}  \cdots { n- \vert N_{k-1}(x) \vert -(m-1)(t-1) \choose t- 1}.
\end{align*}
% Observe that $\vert B_m \cap (A_m^0)^c \vert$ contributes highest order term in $\vert A_m \cap (A_m^0)^c \vert$.
Note that $A_m \cap (A_m^0)^c $ contains at most $m-11$ disjoint hyperedges (here, disjoint means they have only one common vertex that is $v$). Therefore, for large $n$, we get 
\begin{align*}
	 \vert A_m \cap (A_m^0)^c \vert =& \left \vert  \{ \mathcal{H} \in A_m  : \vert H_k^{1,0} (x,v) \vert < m-10 \} \right \vert 
	\\=& { n- \vert N_{k-1}(x) \vert\choose t- 1} \cdots { n- \vert N_{k-1}(x) \vert -(m- 12)(t-1) \choose t- 1} \\
	& { n- \vert N_{k-1}(x) \vert -(m-11)(t-1) \choose t- 2} \cdots 
	\\&{ n- \vert N_{k-1}(x) \vert -(m-1)(t-1)+ 10 \choose t- 2} \left( 1+o(1) \right). 
\end{align*}
Which implies, combining with \eqref{Probability of A m naught}, for $(1-\delta_k)c_k\le m\le (1+\delta_k)c_k$, that 
\begin{align*}
	\Prob (A_m^0 ) &\geq \Prob ( A_m)  \left[1 - \frac { \left \vert  \{ \mathcal{H} \in A_m  : \vert H_k^{1,0} (x,v) \vert < m-10 \} \right \vert}
	{ \left \vert  \{ \mathcal{H} \in A_m  : \vert H_k^{1,0} (x,v) \vert = m \} \right \vert } \right] \\
	&= \Prob ( A_m) \left[1 - \frac {(t-1)^{11} \left( n- \vert N_{k-1}(x) \vert -m(t-1) \right)! (1+o(1))} { \left( n- \vert N_{k-1}(x) \vert -m(t-1) +11 \right)!}
	\right] \\
	&> \Prob ( A_m) \left[1 - \frac {2 (t-1)^{11}} { n^{11} }
	\right],
\end{align*}
as $|N_{k-1}(x)|=o(n)$ as $n\to \infty$. Suppose 
 \[
A^0 = \bigcup\limits_{ m= (1- \delta_1 )c_k}^{ (1+ \delta_1 ) c_k} A_m^0 ,
\]
Then, as a consequence for large $n$, by \eqref{eqn:lem161} we get
\begin{align*}
	\Prob (A^0) 
%	\sum\limits_{ m= (1- \delta_k^1 ){ n- \vert N_{k-1}(x) \vert\choose t- 1}p }^{ (1+ \delta_k^1 ) { n- \vert N_{k-1}(x) \vert\choose t- 1}p} \Prob (A_m) \left [1 - \frac {(t-1)^{11} (1+o(1)) } { \left( n- \vert N_{k-1}(x) \vert -m(t-1) \right)^{11} } \right] \\
%	&\geq \left[1 - \frac {(t-1)^{11} (1+o(1)) } { \left(n- \vert N_{k-1}(x) \vert - (1+ \delta_k^1 ){ n- \vert N_{k-1}(x) \vert\choose t- 1}p(t-1) \right)^ {11} }
%	\right]
%	\sum\limits_{ m= (1- \delta_k^1 ){ n- \vert N_{k-1}(x) \vert\choose t- 1}p }^{ (1+ \delta_k^1 ) { n- \vert N_{k-1}(x) \vert\choose t- 1}p} \Prob (A_m) \\
	&\ge  \Prob (A) \left(1-  \frac{ 2(t-1)^{11} }{n^{11} } \right) 
	\geq \left( 1- 2 n^ { - 12} \right) \left(1-  \frac{ 2(t-1)^{11} }{n^{11} } \right).
\end{align*}    
This completes the proof of the lemma.  
\end{proof}
%As a consequence of the above \cref{lemma 16}, the following corollaries are obtained.
%\begin{remark}    
%Let $1 \leq k \leq d -1$. Then, given $\Omega_{k,x}$, for all vertices $ v \in  \Gamma_{k-1} (x) $,  we obtain 
%\[
%(1- \delta_k^1 ){ n- \vert N_{k-1}(x) \vert\choose t- 1}p \leq \vert  H_k^1 (x,v) \vert \leq (1+ \delta_k^1 ) { n- \vert N_{k-1}(x) \vert\choose t- 1}p , 
%\]
%and
%\[
%\vert H_k^{1,0} (x,v) \vert \geq (1- \delta_k^1 ){ n- \vert N_{k-1}(x) \vert\choose t- 1}p -10 ,
%\] 
%with probability at least
%\[
%1- \vert \Gamma_{k-1} (x) \vert \left[ 2 n^ { - \frac{L}{6(t-1)} } + \frac{ 4(t-1)^{11} }{n^{11} } \right].
%\]
%\end{remark}
%\begin{remark}  
%Let $1 \leq k \leq d -1$. Then for all vertices $ v \in  \Gamma_{k-1} (x) $, we have
%\[
%(t-1) \left[ (1- \delta_k^1 ){ n- \vert N_{k-1}(x) \vert\choose t- 1}p -10 \right] \leq \vert V( H_k^1 (x,v)) \vert \leq (t-1) (1+ \delta_k^1 ){ n- \vert N_{k-1}(x) \vert\choose t- 1}p ,
%\]
%with probability at least
%\[
%1- \vert \Gamma_{k-1} (x)  \vert  \left[ 2 n^ { - \frac{L}{6(t-1)} } + \frac{ 4(t-1)^{11} }{n^{11} } \right] .
%\]
%\end{remark}

\begin{lemma} \label{lemma 17}
Let $t,c,d,n,p$ and  $N$ be in Assumption \ref{ass main}. Let $\delta_1,c_k,L$ be as defined in \Cref{lemma 16}, and $ v_1, v_2,\cdots,v_{\vert \Gamma_{k-1}(x)\vert} \in  \Gamma_{k-1}(x)$ . Then, given $\Omega_{k,x}$, 
\[
(1- \delta_1 )c_k-10 \leq \vert H_k^{1,0} (x,v_i) \vert \leq (1+ \delta_1 )c_k, \mbox{ for every $v_i \in \Gamma_{k-1}(x) $,}
\] 
and at least $ \vert \Gamma_{k-1}(x)\vert -10$ collections of hyperedges are disjoint from the collection of collections of hyperedges $ \{ H_k^1 (x,v_1), H_k^1 (x,v_2),$ $\cdots, H_k^1 (x,v_{\vert \Gamma_{k-1}(x)\vert}) \}$ occur with probability at least $1- 2n^{-10}$.
%\[
%1 - \frac{ 6(t-1)^{11} \vert \Gamma_{k-1} (x) \vert }{n^{11} }.
%\]
\end{lemma}
\begin{proof} 
Let $\mathcal{A} $ denote the set of hypergraphs for which the following hold:
\[
\left( 1- \delta_1 \right)c_k \leq \vert  H_k^1 (x,v_i) \vert \leq \left( 1+ \delta_k^1 \right)c_k \mbox{ and } \vert H_k^{1,0} (x,v_i) \vert \geq (1- \delta_1 )c_k -10 , \forall v_i \in \Gamma_{k-1}(x).
\]
Then by \Cref{lemma 16}, given $\Omega_{ k,x}$, for large $n$ we have 
\begin{align}\label{eqn:lem171}
\Prob ( \mathcal{A}) \geq  1- \vert \Gamma_{k-1} (x)  \vert \frac{ 4(t-1)^{11} }{n^{11} } \geq 1 - \frac{ 1 }{n^{10}}.
\end{align}
Let $a=(1-\delta_1)(n-|N_{k-1}(x)|)p$ and  $b=(1+\delta_1)(n-|N_{k-1}(x)|)p$. Suppose $\ell=|\Gamma_{k-1}(x)|$ and ${\bf m}=(m_1,\ldots, m_{\ell})$ where $m_1,\ldots, m_\ell\in \N$. Suppose
%Let $\mathcal{A}_m $ denote the set of hypergraphs for which 
%$\vert  H_k^1 (x,v_i) \vert =m_i $ and 
%$ \vert H_k^{1,0} (x,v_i) \vert \geq m_i -10 $ hold for all $ v_i \in \Gamma_{k-1}(x)$ , that is,
\[
\mathcal{A}_{\bf m} = \left \{ \mathcal{H} \in \mathscr{H} [n]: \vert  H_k^1 (x,v_i) \vert =m_i \text{ and } \vert H_k^{1,0} (x,v_i) \vert \geq m_i -10 , \forall v_i \in \Gamma_{k-1}(x) \right \} .
\]
Since the events $\mathcal{A}_{\bf m}$ are disjoint, $\mathcal{A} $ is the disjoint union of $\mathcal{A}_{\bf m}$.
Thus, 
\[
\Prob \left( \mathcal{A} \right) = \sum\limits_ {\bf m\in [a,b]^\ell} \Prob \left( \mathcal{A}_{\bf m} \right).
\]
We define $ \vert H_k^{1,0} (x) \vert$, $\mathcal{A}_{\bf m}^0 $ and $B_{\bf m}$ as follows:
\begin{align*}
\vert H_k^{1,0} (x) \vert &= \sup \{ i: \exists i \text { pairwise disjoint sets among } H_k^1 (x,v_1),\ldots , H_k^1 (x,v_l )\},
\\\mathcal{A}_m^0 &= \{ \mathcal{H} \in \mathcal{A}_{\bf m}: \vert H_k^{1,0} (x) \vert \geq \vert \Gamma_{k-1}(x) \vert - 10 \}
\\ \mathcal {B} _{\bf m} &= \{ \mathcal{H} \in \mathcal{A}_{\bf m} \suchthat \vert  H_k^1 (x,v_i) \vert = m_i , \vert H_k^{1,0} (x,v_i) \vert = m_i, \forall v_i \in \Gamma_{k-1}(x) \}.
\end{align*}
Next, given $\Omega_{k,x}$, we calculate the following probability.
\begin{align*}
	\Prob \left( \mathcal{A}_{\bf m}^0 \right) = \Prob \left( \mathcal{A}_{\bf m} \right) \left[1 - \frac {\vert \mathcal{A}_{\bf m} \cap (A_{\bf m}^0)^c \vert} {\vert \mathcal{A}_{\bf m} \vert} \right] 
	\geq \Prob \left( \mathcal{A}_{\bf m} \right) \left[1 - \frac {\vert \mathcal{A}_{\bf m} \cap (\mathcal {A}_{\bf m}^0)^c \vert} {\vert \mathcal {B}_{\bf m} \vert} \right] 
\end{align*}
The last inequality follows because $\mathcal {B}_{\bf m} \subset \mathcal{A}_{\bf m}$. Note that each hypergraph in $\mathcal {B}_{\bf m}$ contains $m_1 + \cdots +m_l$ disjoint hyperedges in $H_k^1 (x)$ (in this case, disjoint means at most they have one element common which is from $\Gamma_{k-1}(x) $). Therefore we have
\begin{align*}
	\vert \mathcal {B}_{\bf m}  \vert = { n- \vert N_{k-1}(x) \vert\choose t- 1} \cdots { n- \vert N_{k-1}(x) \vert -( m_1 + \cdots +m_l -1)(t-1) \choose t- 1} .
\end{align*}
Observe that $\vert \mathcal{A}_{\bf m} \cap ( \mathcal {A}_{\bf m}^0)^c \vert 
= \left \vert \{ \mathcal{H} \in \mathcal{A}_{\bf m}  : \vert H_k^{1,0} (x) \vert < \vert \Gamma_{k-1}(x) \vert - 10 \}  \right \vert$ and $ B_{\bf m} \cap (A_{\bf m}^0)^c \subseteq \mathcal{A}_{\bf m} \cap (\mathcal {A}_{\bf m}^0)^c $. Also, $ \left\vert B_{\bf m} \cap (A_{\bf m}^0)^c \right\vert$ contributes the highest order term in $ \left\vert \mathcal{A}_{\bf m} \cap (\mathcal {A}_{\bf m}^0)^c \right\vert $. It is clear that $\mathcal{A}_{\bf m} \cap (\mathcal {A}_{\bf m}^0)^c$ has at least $11$ collections of hyperedges which have intersection. When each of those $11$ collections of hyperedges have one common vertex, either among themselves or with disjoint collections, we obtain highest order term in the following cardinality. Therefore,
\begin{align*}
	&\vert \mathcal{A}_{\bf m} \cap (\mathcal {A}_{\bf m}^0)^c \vert \\
	=& \left\vert \left \{ \mathcal{H} : \vert  H_k^1 (x,v_i) \vert =m_i , \vert H_k^{1,0} (x,v_i) \vert \geq m_i -10 , \forall v_i \in \Gamma_{k-1}(x), \vert H_k^{1,0} (x) \vert \leq \vert \Gamma_{k-1}(x) \vert - 11  \right \} \right \vert \\
	=& { n- \vert N_{k-1}(x) \vert\choose t- 1}  \cdots
	 { n- \vert N_{k-1}(x) \vert -(m_1 + \cdots +m_l - 12)(t-1) \choose t- 1} \\
	& { n- \vert N_{k-1}(x) \vert -( m_1 + \cdots +m_l - 11)(t-1) \choose t- 2} \cdots 
	\\&{ n- \vert N_{k-1}(x) \vert -(m_1 + \cdots +m_l - 1)(t-1)+10 \choose t- 2} (1+o(1)) .
\end{align*}
Thus from the above we obtain, for large $n$,
\begin{align*}
	\Prob \left( \mathcal{A}_{\bf m}^0 \right) &\geq \Prob \left( \mathcal{A}_{\bf m} \right) \left[1 - \frac {(t-1)^ {11} \left( n- \vert N_{k-1}(x) \vert - (  m_1 + \cdots +m_l)(t-1) \right)! (1+o(1))}
	{ \left( n- \vert N_{k-1}(x) \vert - (  m_1 + \cdots +m_l)(t-1) +11 \right)! }
	\right] \\
	&\geq \Prob \left( \mathcal{A}_{\bf m} \right) \left [1 - \frac {(t-1)^ {11} (1+o(1))}
	{ \left( n- \vert N_{k-1}(x) \vert - (  m_1 + \cdots +m_l) (t-1) \right)^ {11} }
	\right],
	\\& \geq \Prob \left( \mathcal{A}_{\bf m} \right) \left (1 - \frac {2 (t-1)^ {11}}
	{ n^ {11} }
	\right).
\end{align*}
Let $ \mathcal{A}^0  =
\bigcup\limits_ {\bf m\in [a,b]^\ell} \mathcal{A}_{\bf m}^0 $. Then,  by similar arguments as in the proof of \Cref{lemma 16}, 
\begin{align*}
	\Prob \left( \mathcal{A}^0 \right) 
	&  \geq \Prob \left( \mathcal{A} \right) \left (1 - \frac {2 (t-1)^ {11}}
	{ n^ {11} }.
	\right)
%\\	&\geq 1 - \frac{ 6(t-1)^{11} \vert \Gamma_{k-1} (x) \vert }{n^{11} } .
\end{align*}
Which gives the result using \eqref{eqn:lem171}. 
\end{proof}
%\begin{remark} \label{corollary 7}
%We have 
%\begin{align*}
%	(t-1) \left[ \vert \Gamma_{k-1} (x) \vert -10 \right] &\left[ (1- \delta_k^1 ) { n- \vert N_{k-1}(x) \vert\choose t- 1}p -10 \right] \leq  \vert V (H_k^1 (x)) \vert \leq \\
%	&(t-1) \vert \Gamma_{k-1} (x) \vert (1+ \delta_k^1 ){ n- \vert N_{k-1}(x) \vert\choose t- 1}p ,
%\end{align*}
%with probability at least
%\begin{align*}
%	1- \frac{ 6(t-1)^{11} \vert \Gamma_{k-1} (x) \vert }{n^{11} } .
%\end{align*}
%\end{remark}
\begin{lemma} \label{lemma 18}
Let $t,c,d,n,p$ and  $N$ be in Assumption \ref{ass main}. Let $\delta_1$ and $c_k$ be as defined in \Cref{lemma 16}. Define,  for $1 \leq k \leq d-1$,
\[ \epsilon_k = 2\delta_1
%\delta_1 + \frac{10}{\vert \Gamma_{k-1}(x) \vert } + \frac{ 10} {(1- \delta_1 ) c_k} + \frac{(t-1) \vert N_{k-1}(x) \vert}{n-t}.
\]
Then,  given $\Omega_{k,x}$, for a fixed vertex $x$,
\[
\left\vert \vert \Gamma_{k}(x) \vert - (t-1) \vert \Gamma_{k-1}(x) \vert N p \right\vert \leq \epsilon_k (t-1) \vert \Gamma_{k-1}(x) \vert N p,
\]
%equivalently,
%\[
%(1- \epsilon_k ) (t-1) \vert \Gamma_{k-1}(x) \vert N p \leq \vert \Gamma_k(x)  \vert \leq ( 1+ \epsilon_k ) (t-1) \vert \Gamma_{k-1}(x) \vert N p ,
%\]  
holds with probability at least $1-3n^{-10}$.
%\[
%1 - \frac{ 9(t-1)^{11} \vert \Gamma_{k-1} (x) \vert }{n^{11} } .
%\]
\end{lemma}
\begin{proof} [Proof of \cref{lemma 18}]
Given $\Omega_{k,x}$, from \cref{lemma 14} we have 
\begin{align*}
	\vert \Gamma_k(x)  \vert &\leq (1+ \delta_k) \sum_{m=1}^{t-1} (t-m) {\vert \Gamma_{k-1}(x) \vert \choose m} { n- \vert N_{k-1}(x) \vert\choose t- m} p \\
	&\leq ( 1+ \delta_1 ) (t-1) \vert \Gamma_{k-1}(x) \vert { n- \vert N_{k-1}(x) \vert\choose t- 1} p (1+o(1))\\
	%&\leq ( 1+ \delta_1 ) (t-1) \vert \Gamma_{k-1}(x) \vert  N p 
	%\left[ 1 - \frac{(t-1) \vert N_{k-1}(x) \vert}{n-t} \right] \\
	& \leq ( 1+ \epsilon_k ) (t-1) \vert \Gamma_{k-1}(x) \vert N p,
\end{align*}
with probability at least $1-n^{-10}$. On the other hand, from \Cref{lemma 17} we have
\begin{align*}
	\vert \Gamma_{k} (x) \vert &\geq (t-1)  \left( \vert \Gamma_{k-1} (x) \vert -10 \right) \left((1- \delta_1 )c_k -10 \right) \\
%	&\geq (t-1) (1- \delta_k^1 ) \vert \Gamma_{k-1}(x) \vert { n- \vert N_{k-1}(x) \vert\choose t- 1} p 
%	\left[ 1 - \frac{10}{\vert \Gamma_{k-1}(x) \vert } - \frac{ 10}{ (1- \delta_k^1 ) { n- \vert N_{k-1}(x) \vert\choose t- 1} p}
%	\right] \\
%	&\geq (t-1) (1- \delta_k^1 ) \vert \Gamma_{k-1}(x) \vert N p 
%	\left[ 1 - \frac{(t-1) \vert N_{k-1}(x) \vert}{n-t} \right] 
%	\left[ 1 - \frac{10}{\vert \Gamma_{k-1}(x) \vert } - \frac{ 10}{ (1- \delta_k^1 ) { n- \vert N_{k-1}(x) \vert\choose t- 1} p}
%	\right] \\
%	&\geq (t-1) \vert \Gamma_{k-1}(x) \vert N p 
%	\left[ 1 - \delta_k^1 - \frac{10}{\vert \Gamma_{k-1}(x) \vert } - \frac{ 10}{ (1- \delta_k^1 ) { n- \vert N_{k-1}(x) \vert\choose t- 1} p} -- \frac{(t-1) \vert N_{k-1}(x) \vert}{n-t}
%	\right] \\
	& \ge (1 - \epsilon_k ) (t-1) \vert \Gamma_{k-1}(x) \vert N p ,
\end{align*}
with probability at least $1- 2n^{-10}$. Hence the result.
\end{proof}

\begin{lemma} \label{lemma 19}
Let $t,c,d,n,p$ and  $N$ be in Assumption \ref{ass main}.
Set
\[
\eta_k: = \exp \left( \sum\limits_{l=1}^{k} \epsilon_l \right) -1,
\]
where $\epsilon_1,\ldots, \epsilon_k$ are as defined in \Cref{lemma 18}. Define, for $1\le k\le d-1$, 
\begin{equation} \label{bounds of gamma l for all l in hyp}
	\Omega_{k,x}^*:=	\{\mathcal H\in \mathscr H[n]\suchthat \left\vert \vert \Gamma_{l}(x)\vert -  (t-1)^{l} N^l p^l  \right\vert \leq \eta_l (t-1)^{l} N^l p^l ,\forall \; 1 \leq l \leq k\} ,
\end{equation}
Then we have $\Prob(\Omega_{k,x}^*)\ge 1- 3kn^{-10}$  for large $n$.

%\[ 1 - \frac{ 1 }{n^{10} }. \]
\end{lemma}
\begin{proof} [Proof of \cref{lemma 19}]
Recall, $\Omega_{k,x}$ in \eqref{def Omega k,x for hyp}, and it is clear that $ \Omega_{k,x}^* \subseteq \Omega_{k-1, x}^* \subseteq \Omega_{k,x} $. 
We show the result by a recursive relation. We first verify for $k=1$.

%\begin{equation} \label{eq 15}
%	1- \Prob ( \Omega_{k,x}^* ) \leq  \frac{ 18 (t-1)^{11} \vert N_{k-1} (x) \vert }{n^{11} }, \mbox{ for } 1\le k\le d-1.   
%\end{equation}
For $k=1$, $|N_{k-1}(x)|=1$ and the random variable $|H_1(x)|$ has binomial distribution with parameters $N$ and $p$. Therefore, by the  Chernoff's bound we get
\begin{align}\label{eqn:H1up}
\Prob(||H_1(x)|-Np|> \delta_1 Np)\le \frac{1}{n^{10}}, \mbox{  for large $n$}.
\end{align}
On the other hand \Cref{lemma 17} implies that 
\begin{align}\label{eqn:H1lower}
	\Prob(|H_1(x)|\ge (1-\delta_1)c_1 -10)\ge 1- \frac{2}{n^{10}}.
\end{align}
Note that $c_1=Np$. As each hyperedge contributes at most $(t-1)$ points to $|\Gamma_1(x)|$,  by \eqref{eqn:H1up} and \eqref{eqn:H1lower} it follows that 
\begin{align}\label{eqn:k=1}
	||\Gamma_1(x)|-(t-1)Np|\le \eta_1(t-1)Np
\end{align}
holds with probability at least $1-3n^{-10}$. For $k\ge 2$, we have the following relation
\begin{equation} \label{eq 16}
%	\Prob \left( ( \Omega_{k,x}^* )^c \right) = \Prob \left( ( \Omega_{k-1, x}^*)^c \right) + \Prob \left( \left\vert \vert \Gamma_{k}(x)\vert - \{ (t-1)Np \}^k \right\vert \geq \eta_k \{ (t-1)Np \}^k , \Omega_{k-1, x}^* \right).
\Omega_{k,x}^* =\Omega_{k-1 , x}^* \backslash  \left \{ \left\vert \vert \Gamma_{k}(x)\vert - \{ (t-1)Np \}^k \right\vert \geq \eta_k \{ (t-1)Np \}^k , \Omega_{k-1, x}^* \right \}
\end{equation}
On the other hand, given $\Omega_{k-1}^* $, we have 
\begin{align*}
	 \left\vert  (t-1) \vert \Gamma_{k-1}(x) \vert N p-( (t-1)Np )^k  \right\vert 
%	\\&\leq \left\vert \{ (t-1)Np \}^k - (t-1) (1-\eta_{k-1}) (t-1)Np \}^{k-1} N p \right\vert \\
%	&\leq \left\vert \{ (t-1)Np \}^k - (1-\eta_{k-1}) (t-1)Np \}^{k}  \right\vert \\
	&\leq \eta_{k-1}( (t-1)Np )^k.
\end{align*}
Therefore, given $\Omega_{k-1,x}^* $, by the triangle inequality we have
\begin{align*}
	\left\vert \vert \Gamma_{k}(x)\vert - (t-1) \vert \Gamma_{k-1}(x) \vert N p \right\vert + &
	\left\vert ( (t-1)Np )^k  - (t-1) \vert \Gamma_{k-1}(x) \vert N p \right\vert \\
	&\geq \left\vert \vert \Gamma_{k}(x)\vert - ((t-1)Np )^k  \right\vert ,
\end{align*}
Which implies that if
\(
\left\vert \vert \Gamma_{k}(x)\vert - ( (t-1)Np )^k \right\vert \geq \eta_k ( (t-1)Np )^k 
\)
then
\begin{align*}
	\left\vert \vert \Gamma_{k}(x)\vert - (t-1) \vert \Gamma_{k-1}(x) \vert N p \right\vert 
	& \geq \left( \eta_k - \eta_{k-1} \right)
	\{ (t-1)Np \}^k \\
	&= \left[ \exp \left( \sum\limits_{l=1}^{k} \epsilon_l \right) - \exp \left( \sum\limits_{l=1}^{k-1} \epsilon_l \right) \right]
	\{ (t-1)Np \}^k \\
	%&\geq a \left[ \sum\limits_{l=1}^{k} \epsilon_l - \sum\limits_{l=1}^{k-1} \epsilon_l \right] 
	%\{ (t-1)Np \}^k \\
	&\ge \epsilon_k( (t-1)Np )^k, 
\end{align*}
Consequently by \Cref{lemma 18}, given $\Omega_{k,x}$, we have 
\begin{align*}
	&\Prob \left( \left\vert \vert \Gamma_{k}(x)\vert - \left( (t-1)Np \right)^k \right\vert \geq \eta_k \left( (t-1)Np \right)^k , \Omega_{k-1, x}^* \right) 
\\	 & \leq 
	\Prob ( \left\vert \vert \Gamma_{k}(x)\vert - (t-1) \vert \Gamma_{k-1}(x) \vert N p  \right\vert \geq \epsilon_k \left( (t-1)Np \right)^k, \Omega_{k-1, x}^* )
%	&  \leq 
%	\Prob ( \left\vert \vert \Gamma_{k}(x)\vert - (t-1) \vert \Gamma_{k-1}(x) \vert N p  \right\vert \geq \epsilon_k \{ (t-1)Np \}^k ) \\
%%	&=\frac{\Prob \left( \Omega_{k, x} \right)}{\Prob \left( \Omega_{k-1, x}^* \right)}
%	\Prob ( \left\vert \vert \Gamma_{k}(x)\vert - (t-1) \vert \Gamma_{k-1}(x) \vert N p  \right\vert \geq \epsilon_k \{ (t-1)Np \}^k | \Omega_{k, x} )\\
	\leq 
	\frac{3}{n^{10}}.
\end{align*}
%Which further implies
%\[
%\Prob \left( \left\vert \vert \Gamma_{k}(x)\vert - \{ (t-1)Np \}^k \right\vert \geq \eta_k \{ (t-1)Np \}^k , \Omega_{k-1, x}^* \right) \leq
%\frac{ 9(t-1)^{11} \vert \Gamma_{k-1} (x) \vert }{n^{11} }.
%\]
It follows from \eqref{eq 16} and the last inequality that
\begin{align*}
	\Prob \left( ( \Omega_{k,x}^* )^c \right) \leq \Prob \left( ( \Omega_{k-1, x}^*)^c \right) +\frac{3}{n^{10}}.
\end{align*}
Thus, the last equation and  \eqref{eqn:k=1} gives the result.
\end{proof}

\begin{proof} [Proof of \cref{X alpha=1 in hypergraph}]
	Let $H_k(x,y)$ denote the set of all hyperedges which belong to $ H_k(x)$ and contain the vertex $y$, that is,
	\(
	H_k(x,y) = \left\{ e \in H_k(x): y \in e \right\}.
	\)
	Then we have
	\[
	\vert H_k(x,y) \vert = \sum_{m=1}^{t-1} {\vert \Gamma_{k-1}(x) \vert \choose m} { n-1- \vert N_{k-1}(x) \vert\choose t-1- m} p.
	\]
	We take $ k=d$, then given $\Omega_{d-1,x}^*$, we obtain
	\begin{align*}
		\vert H_d(x,y) \vert &= \sum_{m=1}^{t-1} {\vert \Gamma_{d-1}(x) \vert  \choose m} { n-1- \vert N_{d-1}(x) \vert\choose t-1- m}= \frac{(t-1)N\vert \Gamma_{d-1}(x) \vert }{n} (1-o(1)),
	\end{align*}
	as $n\to \infty$. Set 
	\begin{align} \label{definition a,b hyp}
		\begin{split}
		a''&= (1-\eta_{d-1} ) \frac{(t-1)^{d} N^{d} p^{d-1}} {n} (1- o(1)) \text{ and }\\
		b''&= (1+ \eta_{d-1}) \frac{(t-1)^{d} N^{d} p^{d-1} } {n} (1+ o(1)).
		\end{split}
			\end{align}
From the definition of $\Omega_{d-1,x}^*$ in \Cref{lemma 19}, it is clear that
\begin{align} \label{Omega subset Hd}
	\Omega_{d-1,x}^* \subseteq \{ a'' \leq \vert H_d(x,y) \vert \leq b''\}
\end{align}	
Note that similar to \eqref{upperbound of Xalpha} and \eqref{lowerbound of Xalpha}, for $a'',b''$ we have 
\begin{align}
		\Prob \left( X_ \alpha =1 \right) 
		\leq &\sum\limits_ {m\in [a'',b''] } \Prob \left( X_ \alpha =1 \mid \vert H_d(x,y) \vert =m \right) \Prob \left( \vert H_d(x,y) \vert =m \right) \nonumber
\\		&+ \sum\limits_ {m\in [a'',b'']^c } \Prob \left( \vert H_d(x,y) \vert =m    \right) , \label{upperbound of Xalpha hyp}
\\	   \Prob \left( X_ \alpha =1 \right) 
	   \geq & \sum\limits_ {m\in [a'',b''] } \Prob \left( X_ \alpha =1 \mid \vert H_d(x,y) \vert =m \right) \Prob \left( \vert H_d(x,y) \vert =m \right). 
	\label{lowerbound of Xalpha hyp}
\end{align}	
Observe that $	\Prob \left( X_ \alpha =1 \mid \vert H_d(x,y) \vert =m \right) = \Prob \left( d_\mathcal{H} (x,y)>d \text{ in } \mathcal{H} \mid \vert H_d(x,y) \vert =m \right)$, and the distance between the vertices $x$ and $y$ is greater than $d$ if and only if $y$ is not in any hyperedge of $H_d(x,y)$. Therefore
\begin{align} \label{Xalpha given Hd}
	\Prob \left( X_ \alpha =1 \mid \vert H_d(x,y) \vert =m \right) = (1-p)^ { m } . 
\end{align}
Using \eqref{Xalpha given Hd}, \eqref{Omega subset Hd} in \eqref{upperbound of Xalpha hyp} and \eqref{lowerbound of Xalpha hyp}, we obtain
\begin{align} \label{bounds of Xalpha in hyp}
	(1-p)^ { b'' } \Prob \left( \Omega_{d-1,x}^* \right)\leq \Prob \left( X_ \alpha =1 \right) \leq (1-p)^ { a''} + \Prob \left( ( \Omega_{d-1,x}^*) ^c \right)
\end{align}
Also, from \Cref{lemma 19} we have
\begin{align} \label{prob Omega(d-1)}
	\Prob(\Omega_{d-1,x}^*)\ge 1- \frac {3(d-1)} {n^{10}}.
\end{align}
Using the inequality $ e^{- p(p+1)} \leq (1-p) \leq e^{-p} $, the values of $a''$ and $b''$, together with \eqref{prob Omega(d-1)} and Assumption~\ref{ass main} in \eqref{bounds of Xalpha in hyp}
%$\frac{ (t-1)^ {d} N^{d} p^{d}} {n}= \log \left( \frac{n^2}{c} \right)$,
 we obtain
\begin{align*}  
	 \Prob \left( X_ \alpha =1 \right) \leq & \left(\frac{c}{n^2} \right)^ { (1- \eta_{d-1}) (1-o(1)) } +
	\frac{ 3(d-1) }{n^{10} },
\\	\Prob \left( X_ \alpha =1 \right) \geq 
    & \left( \frac{c}{n^2} \right)^ { (1+ \eta_{d-1} ) (p+1) (1+ o(1))}  \left(1 - \frac{ 3(d-1) }{n^{10} } \right).
\end{align*}
Which give the result as $p\log n\to 0$ and $\eta_{d-1}\log n\to 0$ as $n\to \infty$.
\end{proof}

\subsection{Proof of \cref{X alpha=1 beta=1 in hypergraph}} In this subsection, we present the proof of \cref{X alpha=1 beta=1 in hypergraph} arranged in the same manner as the proof of \cref{X alpha=1 in hypergraph}.
\begin{lemma} \label{lemma 21}
	Let $\Omega_{ k,x}^*$ be as defined in \eqref{bounds of gamma l for all l in hyp} and  $\Omega _{k,x,z}^*=\Omega_{k,x}^*\cap \Omega_{ k,z}^*$ for some two vertices $x$ and $z$. 
Then, for $1\le k\le d-1$, for large $n$ we have 
\[
\Prob \left( \Omega _{k,x,z}^* \right) \geq 1- 6k{n^{-10} }.
\]
\end{lemma}
\begin{proof} [Proof of \cref{lemma 21}]
Suppose $\Omega _{k,x,z}^*=\Omega _{k,x}^* \cap \Omega _{k,z}^* $. Then \Cref{lemma 19} implies that 
\[ \Prob \left( (\Omega _{k,x,z}^*)^c \right) \leq \Prob \left( (\Omega _{k,x}^* )^c \right) + \Prob \left( (\Omega _{k,z}^* )^c \right) 
\leq \frac{ 6k }{n^{10} } .
\]
%It is clear that $ \Omega _{k,x,z}^* \subseteq \Omega _{k,x,z} $. Which further implies that
%\begin{equation*} 
%	\Prob ( \Omega _{k,x,z}^c)  \le \frac{ 2 }{n^{10} }.
%\end{equation*}
Hence the result.
\end{proof}

\begin{proof} [Proof of \cref{X alpha=1 beta=1 in hypergraph}]
	Let $\alpha ,\beta \in I$ be such that $\alpha = (x,y)$, $\beta = (z,w)$ with $\alpha\neq \beta$. Next, we define two sets $H_d(x,y)$ and $H_d(z,w)$ as follows:
	\begin{align*}
		H_d(x,y) = \{ e \in H_d(x) : y\in e \} \mbox{ and }
		H_d(z,w) &= \{ e \in H_d(z) :   w\in e \}.
	\end{align*}
Let $m \in \mathbb{N}$, then observe that
\begin{align} \label{Xalpha Xbeta given hd}
	\Prob \left( X_ \alpha =1,X_\beta = 1\mid \vert H_d(x,y) \cup H_d(z,w) \vert =m \right) = \left(1-p\right)^m.
\end{align}	
Similar to \eqref{upperbound of Xalpha} and \eqref{lowerbound of Xalpha}, for any $a,b$ we have
	\begin{align}
			\Prob \left( X_ \alpha =1,X_\beta = 1 \right) 
			\leq & \sum\limits_ {m\in [a,b] } \left(1-p\right)^m \Prob \left( \left\vert  H_d(x,y) \cup H_d(z,w) \right\vert =m \right) \nonumber
	\\	&	+ \sum\limits_ {m\in [a,b]^c } \Prob \left( \vert H_d(x,y) \cup H_d(z,w) \vert =m \right) , \label{upperbound of Xalpha Xbeta hyp}
		\\
		\Prob \left( X_ \alpha =1,X_\beta = 1 \right) 
		\geq & \sum\limits_ {m \in [a,b] } \left(1-p\right)^m \Prob \left( \vert H_d(x,y) \cup H_d(z,w) \vert =m \right) \label{lowerbound of Xalpha Xbeta hyp}
	\end{align}	 
We calculate the required probability by estimating $ \vert H_d(x,y) \cup H_d(z,w) \vert $. To do so, we will consider two cases.

\vspace{.1cm}
\noindent \underline {\bf Case-I :}  Suppose $\{x,y\}\cap \{w,z\}=\emptyset$.
Observe that, given $|\Gamma_{d-1}(x)|$ and $|N_{d-1}(x)|$, we have 
\begin{align*}
	\vert  H_d(x,y) \vert &= \sum_{m=1}^{t-1} {\vert \Gamma_{d-1}(x) \vert \choose m} { n-1- \vert N_{d-1}(x) \vert\choose t-1- m} 
	\\\vert  H_d(z,w) \vert &= \sum_{m=1}^{t-1} {\vert \Gamma_{d-1}(z) \vert \choose m} { n-1- \vert N_{d-1}(z) \vert\choose t-1- m}.
\end{align*}
Note that, given $\Omega_{d-1,x,z}^*$, we have $|\Gamma_{d-1}(x)|=o(n)$ and $|N_{d-1}(x)|=o(n)$. Therefore
\begin{align*}
	\vert  H_d(x,y) \vert=|\Gamma_{d-1}(x)|\frac{(t-1)N}{n}(1+o(1))  \mbox{ and } \vert  H_d(z,w) \vert =|\Gamma_{d-1}(x)|\frac{(t-1)N}{n}(1+o(1)).
\end{align*}
Therefore, by the union bound, we get 
\begin{align}\label{eqn:unionup}
	|H_d(x,y)\cup H_d(z,w)|\le 2|\Gamma_{d-1}(x)|\frac{(t-1)N}{n}(1+o(1)), \mbox{ as $n\to \infty$}.
\end{align}
On the other hand, from the definitions,  it is clear that 
\begin{align*}
	&	H_d(x,y) \cap H_d(z,w) 
	\\=& \{ e \in H_d(x) \cap  H_d(z) : e \cap \Gamma_{d-1}(x) \backslash  \Gamma_{d-1}(z) \neq \emptyset ,e \cap \Gamma_{d-1}(z) \backslash\Gamma_{d-1}(x) \neq \emptyset \text { and }  y, w \in e \} \\
	& \cup \{ e \in H_d(x) \cap  H_d(z) : e \cap \Gamma_{d-1}(x) \cap \Gamma_{d-1}(z) \neq \emptyset \text { and }  y, w \in e \}.
\end{align*}
Which implies that, for $t\ge 3$, as $n\to \infty$,
\begin{align*}
	& \vert  H_d(x,y) \cap H_d(z,w)  \vert \\
	&= \sum\limits_{i=1}^{t-3} \sum\limits_{j=1}^{t-i-2} {\vert \Gamma_{d-1}(x) \backslash \Gamma_{d-1}(z) \vert    \choose i}
	{\vert \Gamma_{d-1}(z) \backslash \Gamma_{d-1}(x) \vert \choose j} 
	{ n- \vert N_{d-1}(x) \cup N_{d-1}(z) \vert -2 \choose t-2-i-j } \\
	&+ \sum\limits_{i=1}^{t-2} {\vert \Gamma_{d-1}(x) \cap \Gamma_{d-1}(z) \vert \choose i}
	{ n- \vert N_{d-2}(x) \cup N_{d-2}(z) \vert-2 \choose t-2-i}
	\\& \leq \l(\vert \Gamma_{d-1}(x) \backslash \Gamma_{d-1}(z) \vert \vert \Gamma_{d-1}(z) \backslash \Gamma_{d-1}(x) \vert \binom{n-1}{t-4}+\vert \Gamma_{d-1}(x) \cap \Gamma_{d-1}(z) \vert\binom{n-1}{t-3}\r)
	\\& (1+o(1)). 
\end{align*}
Since $|\Gamma_{d-1}(x)|=o(n)$ and $N=\binom{n-1}{t-1}$, we get 
\begin{align*}
	\vert  H_d(x,y) \cap H_d(z,w)  \vert\le |\Gamma_{d-1}(x)|\frac{(t-1)^2 N}{n^2}(1+o(1)), \mbox{ as $n\to \infty$.}
\end{align*}
Which implies that, given $ \Omega _{d-1,x,z}^* $, 
\begin{align}\label{eqn:uniondown}
	|H_d(x,y)\cup H_d(z,w)|&=|H_d(x,y)|+|H_d(z,w)|-|H_d(x,y)\cap H_d(z,w)|\nonumber
	\\&\ge 2|\Gamma_{d-1}(x)|\frac{(t-1)N}{n}(1-o(1)),
\end{align}
as $n\to \infty$. 
%\begin{align}
%	\begin{split} \label{definition a,b hyp 2}
%	a ''= 2(1-\eta_{d-1}) \frac { (t-1)^{d} N^d p^{d-1}} {n}(1-o(1)) \\
%	b'' =  2(1+\eta_{d-1}) \frac { (t-1)^{d} N^d p^{d-1}} {n}(1+o(1))
%	\end{split}
%\end{align}
Let $a'',b''$ be as defined \eqref{definition a,b hyp}. Then from the definition of $\Omega _{d-1,x,z}^*$ in \Cref{lemma 21}, we have
\begin{align} \label{Omega subset union Hd}
	\Omega _{d-1,x,z}^* \subseteq \{ 2a'' \leq |H_d(x,y)\cup H_d(z,w)| \leq 2b''\}.
\end{align}
Therefore substituting $a=2a''$ and $b=2b''$ in  \eqref{upperbound of Xalpha Xbeta hyp} , \eqref{lowerbound of Xalpha Xbeta hyp} and by \eqref{Omega subset union Hd} we get
\begin{equation} \label{bounds for X alpha=1, X beta=1 hyp}  	
	\left(1-p\right)^{2b''} \Prob \left( \Omega_{d-1,x,z}^* \right) \leq \Prob \left( X_ \alpha =1,X_\beta = 1 \right)  \leq \left(1-p\right)^{2a''} + \Prob \left( ( \Omega_{d-1,x,z}^*)^c \right) .
	 \end{equation}
Further, the inequality $e^{-p(1-p)}\le 1-p\le e^{-p}$ gives
\begin{align*}
    \left(1-p\right)^{2b''} \geq & e^{- \frac {2(1+\eta_{d-1}) (t-1)^{d} N^d p^{d} (1+p)} {n}(1+o(1))} \\
    \left(1-p\right)^{2a''} \leq & e^{- \frac {2(1-\eta_{d-1}) (t-1)^{d} N^d p^{d}} {n}(1-o(1)) }.	
\end{align*}	
Using the last two inequalities together with \Cref{lemma 21} and Assumption~\ref{ass main} in \eqref{bounds for X alpha=1, X beta=1 hyp}, we obtain 
\begin{align*}
	\Prob \left( X_ \alpha =1, X_\beta = 1 \right) &\leq 
	\left( \frac{c^2}{n^4} \right) ^{ \left( 1 - \eta_{d-1} \right) \left( 1-o(1) \right) } 
	+ \frac{ 6(d-1) }{n^{10} }, \\
	\Prob \left( X_ \alpha =1, X_\beta = 1 \right) &\geq \left( \frac{c^2}{n^4} \right) ^{ \left( 1 + \eta_{d-1} \right) (1+p) \left( 1+o(1) \right) } \left( 1-\frac{ 6(d-1) }{n^{10}} \right).
\end{align*}
Which give the result as $p\log n \to 0$ and $\eta_{d-1}\log n\to 0$ when $n\to \infty$.

\vspace{.1cm}
\noindent \underline {\bf Case-II}:  Suppose $\{x,y\}\cap \{w,z\}\neq \emptyset$. Without loss of generality we assume that $y = w$.   
We consider $H_d(x,y)$ and $H_d(z,y)$.
In this case 
\begin{align*}
	& H_d(x,y) \cap H_d(z,y) \\
	=& \{ e \in H_d(x) \cap  H_d(z) : e \cap \Gamma_{d-1}(x)\backslash \Gamma_{d-1}(z) \neq \emptyset , e \cap \Gamma_{d-1}(z) \backslash \Gamma_{d-1}(x) \neq \emptyset \text { and }  y \in e \} \\
	& \cup \{ e \in H_d(x) \cap  H_d(z) : e \cap \Gamma_{d-1}(x) \cap \Gamma_{d-1}(z) \neq \emptyset \text { and } y \in e \}.
\end{align*}
Which implies that, as $n\to \infty$,
\begin{align*}
	& \vert  H_d(x,y) \cap H_d(z,y)  \vert \\
	&= \sum\limits_{i=1}^{t-2} \sum\limits_{j=1}^{t-i-1} {\vert \Gamma_{d-1}(x) \backslash \Gamma_{d-1}(z) \vert    \choose i}
	{\vert \Gamma_{d-1}(z) \backslash \Gamma_{d-1}(x) \vert \choose j} 
	{ n- \vert N_{d-1}(x) \cup N_{d-1}(z) \vert -1 \choose t-1-i-j } \\
	&+ \sum\limits_{i=1}^{t-1} {\vert \Gamma_{d-1}(x) \cap \Gamma_{d-1}(z) \vert \choose i}
	{ n- \vert N_{d-2}(x) \cup N_{d-2}(z) \vert -1\choose t-1-i}
	\\&=\l(\vert \Gamma_{d-1}(x) \backslash \Gamma_{d-1}(z) \vert \vert \Gamma_{d-1}(z) \backslash \Gamma_{d-1}(x) \vert \binom{n-1}{t-3}+\vert \Gamma_{d-1}(x) \cap \Gamma_{d-1}(z) \vert\binom{n-1}{t-2}\r)(1+o(1)).
\end{align*}
Using the similar arguments as in Case I, we have 
\[
\vert  H_d(x,y) \cap H_d(z,w)  \vert \le |\Gamma_{d-1}(x)|\frac{t^2N}{n^2}(1+o(1)),\;\; \mbox{ as $n\to \infty$}.
\]
After following a similar procedure as in Case-I,  given $ \Omega _{d-1,x,z}^*$, we obtain
\begin{align*}
	\vert  H_d(x,y) \cup  H_d(z,y) \vert &\geq 2|\Gamma_{d-1}(x)|\frac{(t-1)N}{n}(1-o(1)), \;\; \mbox{ as $n\to \infty$}.
\end{align*}
Consequently, by following similar steps as in \underline{Case-I}, we obtain
the result.
		 
\end{proof}

	\section*{Acknowledgement} The research of KA was partially supported by the Inspire Faculty Fellowship: DST/INSPIRE/04/2020/000579. The research of AK was supported by a research fellowship from the University Grants Commission (UGC), Government of India.

%	\bibliographystyle{abbrv}
%	\bibliography{practicebib}

\end{document}